\documentclass[12pt]{article}

\usepackage{amssymb}
\usepackage[cal=boondox]{mathalfa}
\usepackage{amsthm}
\usepackage{amsmath}
\usepackage{graphicx}
\usepackage{fullpage}
\usepackage{color}
\usepackage{enumerate}
 \numberwithin{equation}{section}
\usepackage{booktabs}
\allowdisplaybreaks

\usepackage[pdftex]{hyperref}
 \usepackage{mathtools}
 \mathtoolsset{showonlyrefs}

\theoremstyle{plain}
\newtheorem{thm}{Theorem}[section]
\newtheorem{cor}[thm]{Corollary}

\newtheorem{lem}[thm]{Lemma}
\newtheorem{prop}[thm]{Proposition}

\theoremstyle{definition}
\newtheorem{defn}[thm]{Definition}

\theoremstyle{remark}

\newtheorem{rem}[thm]{Remark}

\newcommand{\N}{\mathbb{N}}
\newcommand{\R}{\mathbb{R}}

\newcommand{\E}{\mathbb{E}}

\newcommand{\bp}{\begin{proof}[\ensuremath{\mathbf{Proof}}]}
\newcommand{\bs}{\begin{proof}[\ensuremath{\mathbf{Solution}}]}
\newcommand{\ep}{\end{proof}}
\newcommand{\id}{\text{id}_{\R}}
\newcommand{\be}{\begin{equation}}
\newcommand{\ee}{\end{equation}}
\newcommand{\sgn}{\text{sgn}}

\begin{document}

\title{A trajectory map for the pressureless Euler equations}

\author{Ryan Hynd}

\maketitle 

\begin{abstract}
We consider the dynamics of a collection of particles that interact pairwise and are restricted to move along the real line.  Moreover, we focus on the situation in which particles undergo perfectly inelastic collisions when they collide.  The equations of motion are a pair of partial differential equations for the particles' mass distribution and local velocity.  We show that solutions of this system exist for given initial conditions by rephrasing these equations in Lagrangian coordinates and then by solving for the associated trajectory map. 
\end{abstract}


\section{Introduction}
In this paper, we will study the dynamics of a collection of particles which interact pairwise and which moves along the real line.  We will also suppose that when particles collide, they undergo perfectly inelastic collisions.   The equations of motion for this type of physical system are the {\it pressureless Euler equations}
\begin{equation}\label{PEE}
\begin{cases}
\hspace{.272in}\partial_t\rho +\partial_x(\rho v)=0\\
\partial_t(\rho v) +\partial_x(\rho v^2)=-\rho(W'*\rho),
\end{cases}
\end{equation}
which hold on $\R\times (0,\infty)$. The first equation expresses the conservation of mass, and the second expresses the conservation of momentum.  Here $\rho$ and $v$ are the respective mass distribution and velocity field of particles and $W$ is the interaction energy.

\par The central goal of this work is to describe how to find a pair $\rho$ and $v$ which solves \eqref{PEE} for given initial conditions 
\be\label{Init}
\rho|_{t=0}=\rho_0\quad \text{and}\quad v|_{t=0}=v_0.
\ee
We typically will assume $\rho_0$ belongs to the space ${\cal P}(\R)$ of Borel probability measures on $\R$ and $v_0: \R\rightarrow \R$ is continuous. 
To this end, we will first produce $X:[0,\infty)\rightarrow L^2(\rho_0)$ which satisfies the {\it pressureless Euler flow equation} 
\be\label{FlowMapEqn}
\dot X(t)=\E_{\rho_0}\left[v_0-\displaystyle\int^t_0(W'*\rho_s)(X(s))ds\bigg| X(t)\right], \quad a.e.\;t\ge 0
\ee
and the initial condition 
\be\label{xInit}
X(0)=\id
\ee
$\rho_0$ almost everywhere.  Here 
\be\label{PushForwardMeasure}
\rho_t:=X(t)_{\#}\rho_0, \quad t\ge 0
\ee
is the push forward of $\rho_0$ under $X(t)$, and $\E_{\rho_0}[ g |X(t)]$ is the conditional expectation of a Borel $g:\R\rightarrow \R$ given $X(t)$. 

\par To emphasize that $X(t)$ is a function on $\R$, we will sometimes write
$$
X(t): \R\rightarrow \R; y\mapsto X(y,t).
$$
The quantity $X(y,t)$ represents the time $t$ position of a particle which was initially at position $y$.  After showing a solution $X$ exists, we will argue that there is a Borel function $v: \R\times [0,\infty)\rightarrow \R$ such that 
\be\label{EulerLagrangeCoord}
\dot X(t)=v(X(t),t), \quad a.e.\;t\ge 0
\ee
$\rho_0$ almost everywhere.  In particular, we will see that 
$$
\rho: [0,\infty)\rightarrow {\cal P}(\R); t\mapsto \rho_t
$$
and $v$ together comprise an appropriately defined weak solution pair for the pressureless Euler system.

\subsection{Main theorem}
Throughout this paper, we will assume the $\rho_0\in {\cal P}(\R)$ has finite second moment
\be
\int_{\R}x^2d\rho_0(x)<\infty
\ee
and 
\be
v_0:\R\rightarrow\R\;\text{ is absolutely continuous}.
\ee
We will also suppose $W:\R\rightarrow \R$ is continuously differentiable, $W$ is even
\be\label{WEven}
W(x)=W(-x), \quad x\in \R
\ee
and $W'$ grows at most linearly 
\be\label{WprimeGrowth}
\sup_{x\in \R}\frac{|W'(x)|}{1+|x|}<\infty.
\ee
Moreover, we will suppose that $W$ is semiconvex. That is, 
\be\label{WSemiCon}
W(x)+\frac{c}{2}x^2\; \text{is convex}
\ee
for some $c>0$.  We recall that concave $W$ corresponds to repulsive interaction between particles. Assuming that $W$ is semiconvex forces $W''(x)\ge -c$ for Lebesgue almost every $x\in \R$, which in a sense limits repulsive interaction. 
 

\begin{thm}\label{mainThm}
 There is a locally Lipschitz continuous $X:[0,\infty)\rightarrow L^2(\rho_0)$ which satisfies the pressureless Euler flow equation  \eqref{FlowMapEqn} and the initial condition \eqref{xInit}. Moreover, $X$ has the following properties. 
\begin{enumerate}[(i)]

\item For Lebesgue almost every $t,s\in [0,\infty)$ with $s\le t$
$$
E(t)\le E(s),
$$
where 
\begin{align*}
E(\tau):=\int_{\R}\frac{1}{2}\dot X(\tau)^2d\rho_0+\int_{\R}\int_{\R}\frac{1}{2}W(X(y,\tau)-X(z,\tau))d\rho_0(y)d\rho_0(z).
\end{align*}

\item For $t\ge 0$ and $y,z\in\textup{supp}(\rho_0)$ with $y\ge z$,
\be
0\le X(y,t)-X(z,t)\le \cosh(\sqrt{c}t)(y-z)+\frac{1}{\sqrt{c}}\sinh(\sqrt{c}t)\int^y_z|v'_0(x)|dx.
\ee

\item For $0<s\le t$ and $y,z\in\textup{supp}(\rho_0)$,
\be
\frac{|X(y,t)-X(z,t)|}{\sinh(\sqrt{c}t)}\le \frac{|X(y,s)-X(z,s)|}{\sinh(\sqrt{c}s)}.
\ee

\end{enumerate}
\end{thm}

\par A few remarks about the statement of this theorem are in order. Locally Lipschitz means that $X:[0,T]\rightarrow L^2(\rho_0)$ is Lipschitz continuous for each $T\ge 0$, and consequently,
\be
\dot X(t)=\lim_{\tau\rightarrow 0}\frac{X(t+\tau)-X(t)}{\tau}
\ee
exists in $L^2(\rho_0)$ for almost almost every $t> 0$. 
The function $E$ in condition $(i)$ represents the total energy of the physical system being modeled by the pressureless Euler flow equation. Condition $(ii)$ asserts that 
$X(t)$ is nondecreasing and absolutely continuous on the support of $\rho_0$
$$
\text{supp}(\rho_0):=\{y\in \R: \rho_0((y-\delta,y+\delta))>0\;\text{for all $\delta>0$}\}.
$$
Property $(iii)$ asserts that $X$ is quantitatively ``sticky."  That is, it quantifies the fact that if $X(y,s)=X(z,s)$, then $X(y,t)=X(z,t)$ for all $t\ge s$.

\par We will show that the existence of a weak solution of \eqref{PEE} for given initial conditions is a corollary of Theorem \ref{mainThm}. In particular, we will verify that $\rho$ defined in \eqref{PushForwardMeasure} and any Borel $v$ which satisfies \eqref{EulerLagrangeCoord} is a weak solution of the pressureless Euler system whose
energy 
\be
\int_{\R}\frac{1}{2}v(x,t)^2d\rho_t(x)+\int_{\R}\int_{\R}\frac{1}{2}W(x-y)d\rho_t(x)d\rho_t(y)
\ee
is essentially nonincreasing in $t$ and which satisfies the one sided Lipschitz condition 
\be
(v(x,t)-v(y,t))(x-y)\le \frac{\sqrt{c}}{\tanh(\sqrt{c}t)}(x-y)^2
\ee
for $\rho_t$ almost every $x,y\in \R$.

\subsection{Prior work}
We have already established the existence of a weak solution pair to the pressureless Euler system with an even, continuously differentiable, semiconvex potential.  In  \cite{Hynd2}, we generated this solution via a Borel probability measure $\eta$ on the space of continuous paths $\Gamma:=C([0,\infty))$ endowed with the topology of local uniform convergence. Specifically, we constructed an $\eta$ which satisfies:  $(i)$ for each bounded, continuous $h:\R\rightarrow \R$ and almost every $t\ge 0$
\be
\int_{\Gamma}\dot\gamma(t) h(\gamma(t))d\eta(\gamma)=\int_{\Gamma}\left[v_0(\gamma(0))-\int^t_0(W'*\rho_s)(\gamma(s))ds \right]h(\gamma(t))d\eta(\gamma),
\ee
where $\rho_s\in {\cal P}(\R)$ is defined via
$$
\int_{\R}h(x)d\rho_s(x)=\int_{\Gamma}h(\gamma(s))d\eta(\gamma);
$$ 
$(ii)$ there is a Borel $v:\R\times[0,\infty)\rightarrow \R$, such that
$$
\dot\gamma(t)=v(\gamma(t),t),\quad\text{a.e.}\;t>0
$$
for $\eta$ almost every $\gamma\in \Gamma$.   Then we checked that $\rho:[0,\infty)\rightarrow {\cal P}(\R); t\mapsto \rho_t$ and $v$ is indeed a weak solution pair. 

\par Along the way, we derived some specific information on $\eta$ such as it is concentrated on absolutely continuous paths, it satisfies various energy estimates, and 
\be
\frac{|\gamma(t)-\xi(t)|}{\sinh(\sqrt{c}t)}\le \frac{|\gamma(s)-\xi(s)|}{\sinh(\sqrt{c}s)}.
\ee
for $0<s\le t$ and $\eta$ almost every $\gamma,\xi\in\Gamma$. We consider Theorem \ref{mainThm} to be a
refinement of the main result in \cite{Hynd2} as it tells us that we can choose $\eta$ as the push forward of $\rho_0$ under the map $\R\mapsto \Gamma; y\mapsto X(y,\cdot)$. Here $X(y,\cdot)$ is the path $t\mapsto X(y,t)$, which is continuous for $\rho_0$ almost every $y\in\R$.  That is, $\eta$ can be specified as
$$
\int_{\Gamma}F(\gamma)d\eta(\gamma)=\int_{\R}F(X(y,\cdot))d\rho_0(y)
$$
for each $F:\Gamma\rightarrow \R$ that is continuous and bounded.

\par There have been many other works on pressureless Euler type systems in one spatial dimension. 
Especially since they are special cases of the multidimensional systems of equations which arise in the study of galaxy formation \cite{Gurbatov, Zeldovich}. One of the early mathematical works on this topic was by E, Rykov and Sinai \cite{ERykovSinai}, where they studied the case $W(x)=|x|$ which corresponds to gravitational interaction between a collection of interacting particles constrained to move along the real line. We acknowledge that the existence of solutions for this particular case does not follow from Theorem \ref{mainThm} as $W(x)=|x|$ isn't continuously differentiable. Nevertheless, we will revisit this particular case below. 

\par Another very influential study on this topic was done by Brenier, Gangbo, Savar\'e and Westdickenberg \cite{BreGan}. In comparison to our work, they considered general interactions which could be attractive or repulsive. We also note that they recast the pressureless Euler equations in another coordinate system, and they were able to obtain precise information about solutions from the resulting differential inclusions.  Other work with related approaches were done by Gangbo, Nguyen, and Tudorascu \cite{GNT} and Nguyen and Tudorascu \cite{NguTud} on the Euler-Poisson system and by Brenier and Grenier \cite{BreGre},  Natile and Savar\'e \cite{NatSav}, and 
Cavalletti, Sedjro and Westdickenberg \cite{MR3296602} for the sticky particle system ($W\equiv 0$ in \eqref{PEE} or equation \eqref{SPS} below).  We also recommend the additional references \cite{Bezard,Liu,Guo,Jabin,Jin,LeFloch,Makino,Shen} for results on stationary solutions, local existence, uniqueness, and hydrodynamic limits related to pressureless Euler type systems.

\par The particular approach we take in this paper is motivated by the work of Dermoune \cite{Dermoune} on the sticky particle system 
\begin{equation}\label{SPS}
\begin{cases}
\hspace{.272in}\partial_t\rho +\partial_x(\rho v)=0\\
\partial_t(\rho v) +\partial_x(\rho v^2)=0.
\end{cases}
\end{equation}
In particular, Dermoune was the first to identify that 
\be\label{SPSFlowMapEqn}
\dot X(t)=\E_{\rho_0}\left[v_0| X(t)\right], \quad a.e.\;t\ge 0
\ee
is the natural equation for the sticky particle system in Lagrangian variables.   We performed a thorough analysis of \eqref{SPSFlowMapEqn} in \cite{MR4007615} and regard Theorem \ref{mainThm} as a significant generalization of the main results of \cite{MR4007615}.

\subsection{Euler-Poisson equations}
 As mentioned above, we will also consider the {\it Euler-Poisson equations} 
\be\label{EP1D}
\begin{cases}
\hspace{.272in}\partial_t\rho +\partial_x(\rho v)=0\\
\partial_t(\rho v) +\partial_x(\rho v^2)=-\rho(\sgn*\rho).
\end{cases}
\ee
Here 
\be
\sgn(x):=
\begin{cases}
1,\quad &x>0\\
0,\quad &x=0\\
-1,\quad &x<0,
\end{cases}
\ee
and the associated interaction potential is $W(x)=|x|$. This system governs the dynamics of a collection of particles in which the force on each particle is proportional to the total mass to the right of the particle minus the total mass to the left of the particle; when particles collide, they undergo perfectly inelastic collisions \cite{BreGan, ERykovSinai}. This is a simple model for gravitationally interacting particles which are constrained to move on the real line. 

\par As we did for the pressureless Euler system,  we will design a trajectory mapping $X$ which satisfies the {\it Euler-Poisson flow equation}
\be\label{FlowMapEqnEP}
\dot X(t)=\E_{\rho_0}\left[v_0-\displaystyle\int^t_0(\text{sgn}*\rho_s)(X(s))ds\bigg| X(t)\right], \quad a.e.\;t\ge 0
\ee
and the initial condition \eqref{xInit}. However, since $\sgn$ is not continuous, we will have to argue a bit differently than we did to prove Theorem \ref{mainThm} in order to obtain the following theorem. 

\begin{thm}\label{secondThm}
There is a Lipschitz continuous $X:[0,\infty)\rightarrow L^2(\rho_0)$ which satisfies the Euler-Poisson flow equation \eqref{FlowMapEqnEP} and the initial condition \eqref{xInit}. Moreover, $X$ has the following properties. 
\begin{enumerate}[(i)]

\item For Lebesgue almost every $t,s\in [0,\infty)$ with $s\le t$
$$
E(t)\le E(s),
$$
where 
\be
E(\tau):=\int_{\R}\frac{1}{2}\dot X(\tau)^2d\rho_0+\int_{\R}\int_{\R}\frac{1}{2}|X(y,\tau)-X(z,\tau)|d\rho_0(y)d\rho_0(z).
\ee

\item For $t\ge 0$ and $y,z\in\textup{supp}(\rho_0)$ with $y\ge z$,
\be
0\le X(y,t)-X(z,t)\le y-z+t\int^y_z|v'_0(x)|dx.
\ee

\item For $0<s\le t$ and $y,z\in\textup{supp}(\rho_0)$,
\be
\frac{1}{t}|X(y,t)-X(z,t)|\le \frac{1}{s}|X(y,s)-X(z,s)|.
\ee

\end{enumerate}
\end{thm}

\par As with the pressureless Euler system, we will be able to generate a weak solution pair of the Euler-Poisson system a the solution $X$ obtained in Theorem \ref{secondThm}. Namely, $\rho$ defined in \eqref{PushForwardMeasure} and any Borel $v$ which satisfies \eqref{EulerLagrangeCoord} is a weak solution of the  Euler-Poisson system whose
total energy 
\be
\int_{\R}\frac{1}{2}v(x,t)^2d\rho_t(x)+\int_{\R}\int_{\R}\frac{1}{2}|x-y|d\rho_t(x)d\rho_t(y)
\ee
is nonincreasing in time and which fulfills the ``entropy" inequality
\be
(v(x,t)-v(y,t))(x-y)\le \frac{1}{t}(x-y)^2
\ee
for $\rho_t$ almost every $x,y\in \R$. 
\\
\par The organization of this paper is as follows. First, we will review a few preliminaries needed for our study in section \ref{PrelimSect}. Then we will show by a near explicit construction how to solve the pressureless Euler flow equation when the support of $\rho_0$ is finite in section \ref{StickyParticleSect}. In section \ref{CompactSect}, we will analyze these special solutions and show they are compact in a certain sense. This compactness will allow us to solve the pressureless Euler flow equation for a general $\rho_0$ and consequently to solve the pressureless Euler equations for given initial conditions.  Finally, in section \ref{EPsect}, we will show how to alter the arguments we used for the pressureless Euler flow equation to solve the flow equation associated with the Euler-Poisson equations.

\section{Preliminaries}\label{PrelimSect}
In this section, we will briefly recall the facts we will need regarding the convergence of probability measures and conditional expectation. 

\subsection{Convergence of probability measures}
As in the introduction, we denote ${\cal P}(\R^d)$ as the space of Borel probability measures on $\R^d$. We will also write $C_b(\R^d)$ for the space of bounded continuous functions on $\R^d$. We will say that a sequence $(\mu^k)_{k\in\N}\subset {\cal P}(\R^d)$ converges to $\mu$ in ${\cal P}(\R^d)$ {\it narrowly} provided 
\be\label{gConvUndermukay}
\lim_{k\rightarrow\infty}\int_{\R^d}gd\mu^k=\int_{\R^d}gd\mu
\ee
for every $g\in C_b(\R^d)$.  It turns out that $(\mu^k)_{k\in\N}$ converges to $\mu$ narrowly if and only if $\lim_{k\rightarrow\infty}\mathcal{d}(\mu,\mu^k)=0$, where $\mathcal{d}$ is a metric of the form
\be\label{NarrowMetric}
\mathcal{d}(\mu,\nu):=\sum^\infty_{j=1}\frac{1}{2^j}\left|\int_{\R^d}h_jd\mu-\int_{\R^d}h_jd\nu\right|,\quad \mu,\nu\in {\cal P}(\R^d).
\ee
 Here each $h_j:\R^d\rightarrow \R$ satisfies 
\be
|h_j(x)|\le 1\quad \text{and}\quad |h_j(x)-h_j(y)|\le |x-y|
\ee
for $x,y\in \R^d$ (Remark 5.1.1 of \cite{AGS}). Furthermore, $({\cal P}(\R^d), \mathcal{d})$ is a complete metric space.

\par We will need to be able to identify when a sequence of measures in ${\cal P}(\R^d)$ has a narrowly convergent subsequence.  Fortunately, Prokhorov's theorem provides a necessary and sufficient condition; it asserts that  $(\mu^k)_{k\in\N}\subset {\cal P}(\R^d)$ has a narrowly convergent subsequence if and only if there is $\varphi: \R^d\rightarrow [0,\infty]$ with compact sublevel sets such that
\be\label{prohorovCond}
\sup_{k\in \N}\int_{\R^d}\varphi d\mu^k<\infty
\ee
(Theorem 5.1.3 of \cite{AGS}). In addition, we will need to know when \eqref{gConvUndermukay} holds for unbounded $g$. It turns out that if $g: \R^d\rightarrow \R$ is continuous and
$$
\lim_{R\rightarrow\infty}\int_{|g|\ge R}|g|d\mu^k=0
$$
uniformly in $k\in \N$, then \eqref{gConvUndermukay} holds (Lemma 5.1.7 of \cite{AGS}).  In this case, we say that  $|g|$ is {\it uniformly integrable} with respect to the sequence $(\mu^k)_{k\in\N}$.

\par The following lemma will also prove to be useful.

\begin{lem}[Lemma 2.1 of \cite{MR4007615}]\label{LemmaVariantNarrowCon}
Suppose $(g^k)_{k\in\N}$ is a sequence of continuous functions on $\R^d$ which converges locally uniformly to $g$ and $(\mu^k)_{k\in\N}\subset {\cal P}(\R^d)$ converges narrowly to $\mu$. Further assume there is $h:\R^d\rightarrow [0,\infty)$ with compact sublevel sets, which is uniformly integrable with respect to  $(\mu^k)_{k\in\N}$ and  satisfies
\be\label{gkaydominatedbyH}
|g^k|\le h
\ee
for each $k\in \N$. Then
\be\label{UniformonvUndermukay}
\lim_{k\rightarrow\infty}\int_{\R^d}g^kd\mu^k=\int_{\R^d}gd\mu.
\ee
\end{lem}

\subsection{The push-forward}
Suppose $f:\R^d\rightarrow \R^n$ is Borel measurable and $\mu\in {\cal P}(\R^d)$. We define the {\it push-forward} of $\mu$ through $f$ as the probability measure $f_{\#}\mu\in {\cal P}(\R^n)$ which satisfies 
$$
\int_{\R^n}g(y)d(f_{\#}\mu)(y)=\int_{\R^d}g(f(x))d\mu(x)
$$
for every $g\in C_b(\R^n)$.  We note 
$$
f_{\#}\mu(A)=\mu(f^{-1}(A))
$$
for Borel $A\subset \R^n$.  Moreover, if $f$ is continuous and $\mu^k\rightarrow \mu$ narrowly in ${\cal P}(\R^d)$, then 
$$
f_{\#}\mu^k\rightarrow f_{\#}\mu
$$ 
in ${\cal P}(\R^n)$.

\subsection{Conditional expectation}
Suppose $\mu\in {\cal P}(\R)$, $g\in L^2(\mu)$ and $Y: \R\rightarrow \R$ is a Borel measurable function.
A {\it conditional expectation} of $g$ with respect to $\mu$ given $Y$ is an $L^2(\mu)$ function $\E_{\mu}[g|Y]$ which satisfies two conditions: $(i)$
\be\label{IntegralCondExpCond}
\displaystyle\int_{\R}\E_\mu[g|Y]\; h(Y)d\mu=\int_{\R}g\; h(Y)d\mu
\ee
for each Borel $h:\R\rightarrow \R$ with $h(Y):=h\circ Y\in L^2(\mu)$; and $(ii)$
$$
\displaystyle\E_{\mu}[g|Y]=f(Y)
$$
$\mu$ almost everywhere for a Borel $f:\R\rightarrow \R$ with $f(Y)\in L^2(\mu)$.  The existence and $\mu$ almost everywhere uniqueness of a conditional expectation can be proved using the Radon-Nikodym theorem.

\par We emphasize that $X:[0,\infty)\rightarrow L^2(\rho_0)$ satisfies the pressureless Euler flow equation \eqref{FlowMapEqn}, provided the following two conditions hold for almost every $t\ge 0$: $(i)$
\be
\int_{\R}g(X(t))\dot X(t)d\rho_0=\int_{\R}g(X(t))\left[v_0-\displaystyle\int^t_0(W'*\rho_\tau)(X(\tau))d\tau\right]d\rho_0
\ee
for each $g\in C_b(\R)$; and $(ii)$ there exists a Borel $u: \R\rightarrow \R$ for which 
$$
\dot X(t)=u(X(t))
$$
$\rho_0$ almost everywhere.

\section{Sticky particle trajectories}\label{StickyParticleSect}
In this section, we will assume that $\rho_0$ is a convex combination of Dirac measures
\be\label{discreterhozero}
\rho_0:=\sum^N_{i=1}m_i\delta_{x_i}\in {\cal P}(\R).
\ee
In particular, we suppose that $x_1,\dots, x_N\in \R$ are distinct and $m_1,\dots, m_N>0$ with $\sum^N_{i=1}m_i=1$.  We also define  
$$
v_i:=v_0(x_i)
$$
for $i=1,\dots, N$.  It turns out that there is a natural ODE system related to the pressureless 
Euler flow equation, which is
\be\label{gammaODEt}
\ddot \gamma_i(t)=-\sum^N_{j=1}m_jW'(\gamma_i(t)-\gamma_j(t)).
\ee
These are Newton's equations for $N$ interacting particles with masses $m_1,\dots, m_N$; the positions of these particles  
are described by the trajectories $\gamma_1,\dots, \gamma_N$.

\par It turns out that a solution of the pressureless Euler flow equation can be built from these particle trajectories by first setting 
$$
X(x_i,t)=\gamma_i(t), \quad t\ge 0.
$$ 
However, when trajectories intersect, we must 
modify the paths. Remarkably, the natural thing to do is to require that the corresponding particles undergo perfectly 
inelastic collisions when they collide. This amounts to requiring that the trajectories coincide and 
that their slopes average from the moment  they intersect. On any time interval when no collisions 
occur, the resulting trajectories will satisfy \eqref{gammaODEt}.  We will call these paths {\it sticky particle trajectories}
and we shall see that they are the building blocks for more general solutions. 

\begin{figure}[h]
\centering
 \includegraphics[width=.8\textwidth]{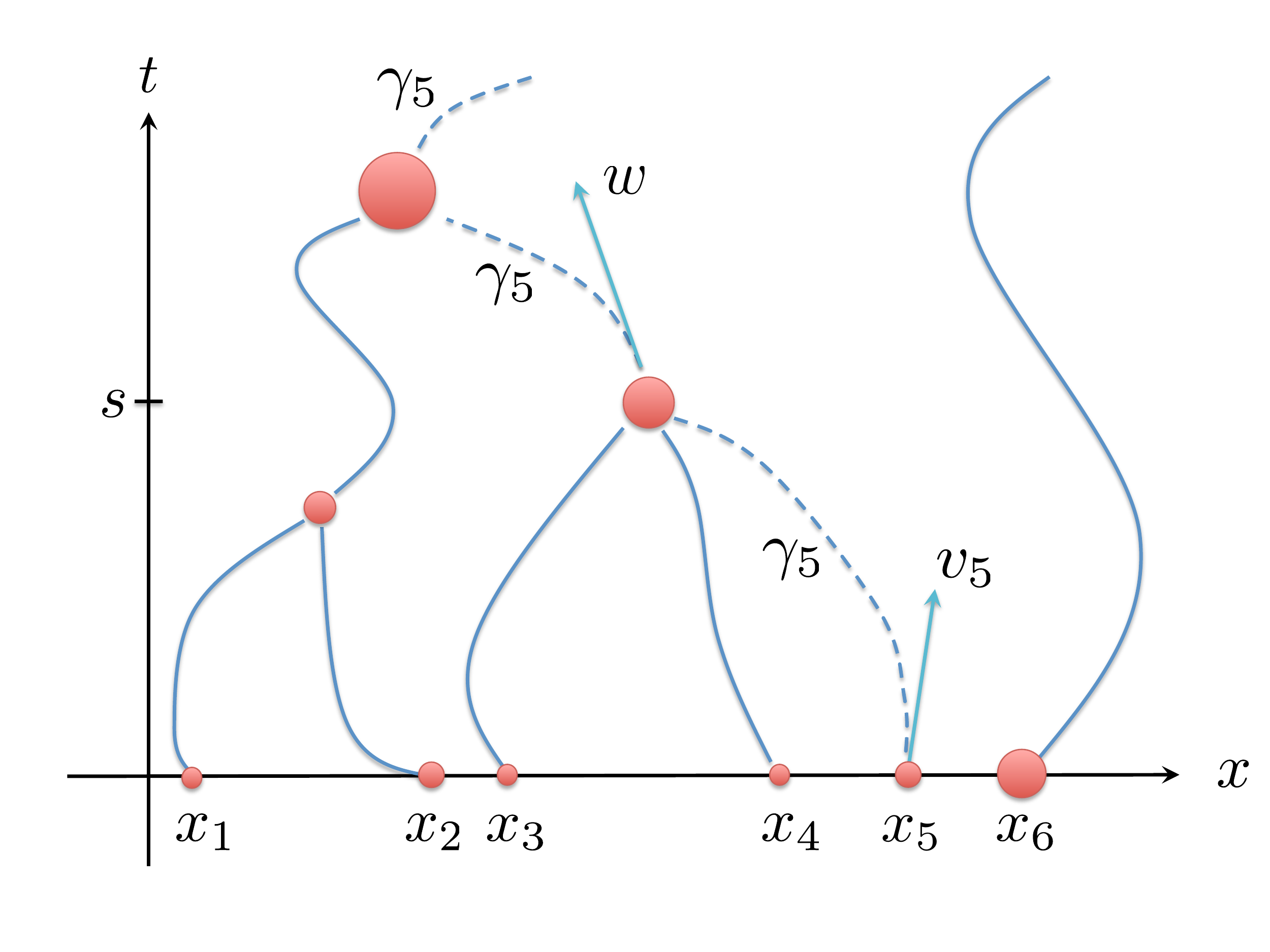}
 \caption{A schematic of sticky particle trajectories for $N=6$. We have indicated the starting positions $x_1, \dots, x_6$ on the real line and we sketched the corresponding point masses larger than points to emphasize that their masses may be distinct. The path $\gamma_5$ that tracks point mass $m_5$ is shown in dashed along with its initial velocity $v_5$. Observe that there is a collision at time $s$ between the point masses $m_3, m_4$ and $m_5$. As a result, the slope $w$ satisfies $(m_3+m_4+m_5)w=m_3\dot\gamma_3(s-)+m_4\dot\gamma_4(s-)+m_5\dot\gamma_5(s-)$.
 }\label{Fig1}
\end{figure}

\par The following proposition asserts that these trajectories exist and satisfy a few basic properties. 

\begin{prop}[Proposition 2.1 \cite{Hynd2}]\label{StickyParticlesExist}
There are continuous, piecewise $C^2$ paths 
$$
\gamma_1,\dots,\gamma_N : [0,\infty)\rightarrow \R
$$
with the following properties. \\
(i) For $i=1,\dots, N$ and all but finitely many $t\in (0,\infty)$, \eqref{gammaODEt} holds. \\
(ii) For $i=1,\dots, N$,
$$
\gamma_i(0)=x_i\quad \text{and}\quad \dot\gamma_i(0+)=v_i.
$$
(iii) For $i,j=1,\dots, N$, $0\le s\le t$ and $\gamma_i(s)=\gamma_j(s)$ imply 
$$
\gamma_i(t)=\gamma_j(t).
$$
(iv) If $t>0$, $\{i_1,\dots, i_k\}\subset\{1,\dots, N\}$, and
$$
\gamma_{i_1}(t)=\dots=\gamma_{i_k}(t)\neq \gamma_i(t)
$$
for $i\not\in\{i_1,\dots, i_k\}$, then
$$
\dot\gamma_{i_j}(t+)=\frac{m_{i_1}\dot\gamma_{i_1}(t-)+\dots+m_{i_k}\dot\gamma_{i_k}(t-)}{m_{i_1}+\dots+m_{i_k}}
$$
for $j=1,\dots, k$.
\end{prop}
\begin{rem}\label{interchangeRemark}
 Using property $(i)$, it is routine to check that $\dot\gamma_i(t\pm)$ both exist for each $t>0$ and $i=1,\dots, N$. Moreover, 
$$
\dot\gamma_i(t\pm)=\lim_{h\rightarrow 0^\pm}\frac{\gamma_i(t+h)-\gamma_i(t)}{h}.
$$
\end{rem}

\par A corollary of property $(iv)$ above is the what we call the {\it averaging property}. It is a general assertion about the conservation 
of momentum and is stated as follows.
\begin{cor}[Proposition 2.6 of \cite{Hynd2}]\label{AveragingProp}
Suppose $g:\R\rightarrow \R$ and $0\le s<t$. Then 
\begin{align}\label{AveragingProp}
\sum^N_{i=1}m_ig(\gamma_i(t))\dot\gamma_i(t+)= \sum^N_{i=1}m_ig(\gamma_i(t))\left[\dot\gamma_i(s+)-\int^t_s\left(\sum^N_{j=1}m_jW'(\gamma_i(\tau)-\gamma_j(\tau))\right)d\tau \right].  \;\;\;
\end{align}
\end{cor}

\subsection{Quantitative stickiness }
Recall our standing assumption that there is a constant $c>0$ chosen so that $W(x)+(c/2)x^2$ is convex.  In terms of this constant,  
we can quantify $(iii)$ in Proposition \ref{StickyParticlesExist}. Namely, we can estimate the distance $|\gamma_i(t)-\gamma_j(t)|$ in terms of 
the distance $|\gamma_i(s)-\gamma_j(s)|$ for $s\le t$. This is why we call the following assertion the {\it quantitative sticky particle property}.
 \begin{prop}[Proposition 2.5 of \cite{Hynd2}]\label{PropQSPP}
For each $i,j=1,\dots,N$ and $0<s\le t$
\be
\frac{|\gamma_i(t)-\gamma_j(t)|}{\sinh(\sqrt{c}t)}
\le \frac{|\gamma_i(s)-\gamma_j(s)|}{\sinh(\sqrt{c}s)}.
\ee
\end{prop}

\par An immediate corollary is as follows. 
\begin{prop}\label{measurabilityProp}
For each $0<s\le t$, there is a function $f_{t,s}:\R\rightarrow \R$ for which 
$$
\gamma_i(t)=f_{t,s}(\gamma_i(s))
$$
for $i=1,\dots, N$ and 
\be\label{fteessLip}
|f_{t,s}(x)-f_{t,s}(y)|\le \frac{\sinh(\sqrt{c}t)}{\sinh(\sqrt{c}s)}|x-y|
\ee
for $x,y\in \R.$ 
\end{prop}

\begin{proof}
By property $(iii)$ of Proposition \ref{StickyParticlesExist}, the cardinality of the set
$$
\{\gamma_1(t),\dots,\gamma_N(t)\}
$$
is nonincreasing in $t$. It follows that there is a surjective function 
$$
g_{t,s}:\{\gamma_1(s),\dots,\gamma_N(s)\}\rightarrow \{\gamma_1(t),\dots,\gamma_N(t)\}; \gamma_i(s)\mapsto \gamma_i(t)
$$
for $0< s\le t$. By the quantitative sticky particle property, $g_{t,s}$ satisfies the Lipschitz condition \eqref{fteessLip}. We can then extend $g_{t,s}$ to all of $\R$ in order to obtain the desired Lipschitz function $f_{t,s}$. 
\end{proof}

\subsection{Energy estimates}
Sticky particle trajectories have nonincreasing energy. That is, 
\be\label{ContNonincreaseEnergy}
\frac{1}{2}\sum^N_{i=1}m_i\dot\gamma_i(t+)^2+
\frac{1}{2}\sum^N_{i,j=1}m_im_jW(\gamma_i(t)-\gamma_j(t))\le \frac{1}{2}\sum^N_{i=1}m_i\dot\gamma_i(s+)^2+
\frac{1}{2}\sum^N_{i,j=1}m_im_jW(\gamma_i(s)-\gamma_j(s))
\ee
for $0\le s<t$ (Proposition 2.8 of \cite{Hynd2}).  Using the semiconvexity of $W$, we can derive the subsequent kinetic energy estimates.  We will express this result in terms of the increasing function 
$$
\vartheta(t):=e^{(c+1)t^2}\int^t_{0}e^{-(c+1)s^2}ds,\quad t\ge 0.
$$

\begin{lem}\label{EnergyEst}
For each $t\ge 0$,
\be\label{DiscreteEnergyEst1}
\int^t_{0}\sum^N_{i=1}m_i\dot\gamma_i(s)^2ds\le \left(\sum^N_{i=1}m_iv_i^2+\frac{1}{2}\sum^N_{i,j=1}m_im_jW'(x_i-x_j)^2 \right)\vartheta(t).
\ee
And for all but finitely many $t\ge 0$,
\be\label{DiscreteEnergyEst2}
\sum^N_{i=1}m_i\dot\gamma_i(t)^2\le \left(\sum^N_{i=1}m_iv_i^2+\frac{1}{2}\sum^N_{i,j=1}m_im_jW'(x_i-x_j)^2 \right)\vartheta'(t).
\ee
\end{lem}

\begin{proof}
Due to the convexity of $x\mapsto W(x)+(c/2)x^2$,\begin{align*}
W(\gamma_i(t)-\gamma_j(t))&\ge W(x_i-x_j)+ W'(x_i-x_j)(\gamma_i(t)-x_i-(\gamma_j(t)-x_j))\\
&\hspace{2in} -\frac{c}{2}(\gamma_i(t)-x_i-(\gamma_j(t)-x_j))^2 \\
&\ge W(x_i-x_j)- \frac{1}{2}W'(x_i-x_j)^2  -\frac{c+1}{2}(\gamma_i(t)-x_i-(\gamma_j(t)-x_j))^2 \\
&\ge W(x_i-x_j)- \frac{1}{2}W'(x_i-x_j)^2  -(c+1)((\gamma_i(t)-x_i)^2+(\gamma_j(t)-x_j)^2) \\
&\ge W(x_i-x_j)- \frac{1}{2}W'(x_i-x_j)^2 -(c+1)t\left( \int^t_{0}\dot\gamma_i(s)^2ds+\int^t_{0}\dot\gamma_j(s)^2ds \right).
\end{align*}
Combining these lower bounds with \eqref{ContNonincreaseEnergy} at $s=0$ gives
\be\label{SteptoSecondGronwall}
\sum^N_{i=1}m_i\dot\gamma_i(t)^2\le \sum^N_{i=1}m_iv_i^2+\frac{1}{2}\sum^N_{i,j=1}m_im_jW'(x_i-x_j)^2+2(c+1)t\int^t_{0}\sum^N_{i=1}m_i\dot\gamma_i(s)^2ds
\ee
for all but finitely many $t\ge 0$.  As a result,
\begin{align*}
&\frac{d}{dt}e^{-(c+1)t^2}\int^t_{0}\sum^N_{i=1}m_i\dot\gamma_i(s)^2ds\\
&\quad =e^{-(c+1)t^2}\left(\sum^N_{i=1}m_i\dot\gamma_i(t)^2-2(c+1) t\int^t_{0}\sum^N_{i=1}m_i\dot\gamma_i(s)^2ds\right)\\
&\quad \le e^{-(c+1)t^2}\left(\sum^N_{i=1}m_iv_i^2 +\frac{1}{2}\sum^N_{i,j=1}m_im_jW'(x_i-x_j)^2\right)
\end{align*}
for all but finitely many $t\ge 0$. We can then integrate from $0$ to $t$ to derive  \eqref{DiscreteEnergyEst1}. Inequality \eqref{DiscreteEnergyEst2} follows from \eqref{DiscreteEnergyEst1} and \eqref{SteptoSecondGronwall}.
\end{proof}

\subsection{Stability estimate}
We need one more estimate that depends on the following elementary lemma. 

\begin{lem}\label{ODEineqLemma2}
Suppose $T>0$ and $y:[0,T)\rightarrow \R$ is continuous and piecewise $C^2$. Further assume 
\be\label{slopeDecreaseY}
\dot y(t+)\le \dot y(t-)
\ee
for each $t\in (0,T)$ and that there is $c>0$ for which
\be\label{YdoubleprimelessthanY}
\ddot y(t)\le c y(t)
\ee
for all but finitely many $t\in (0,T)$.  Then 
\be
y(t)\le \cosh(\sqrt{c}t)y(0)+\frac{1}{\sqrt{c}}\sinh(\sqrt{c}t)\dot y(0+).
\ee
for $t\in [0,T)$. 
\end{lem}

\begin{proof} By a routine scaling argument, it suffices to verify this assertion for $c=1$.  To this end, we suppose in addition that there are times $0<t_1<\dots<t_n$ for which $y$ is $C^2$ on the intervals $(0,t_1),\dots, (t_n,T)$. 

\par Define
$$
u(t):=\frac{y(t)}{\cosh(t)}, \quad t\in (0,T)
$$
so that $y(t)=u(t)\cosh(t).$ Observe
$$
\dot u(t)=\frac{\dot y(t)}{\cosh(t)}-\frac{y(t)}{\cosh(t)^2}\sinh(t)
$$
and 
\begin{align*}
\ddot y(t) & = \ddot u(t)\cosh(t)+2\dot u(t)\sinh(t)+u(t)\cosh(t)\\
&=\ddot u(t)\cosh(t)+2\dot u(t)\sinh(t)+y(t)\\
&\le y(t)
\end{align*}
for $t\in (0,T)\setminus\{t_1,\dots, t_n\}$. Consequently, 
\be\label{UdotLessThan}
\frac{d}{dt}\left(\dot u(t)\cosh(t)^2\right)=\cosh(t)\left(\ddot u(t)\cosh(t)+2\dot u(t)\sinh(t)\right)\le 0
\ee
for $t\in (0,T)\setminus\{t_1,\dots, t_n\}$.

\par In view of \eqref{UdotLessThan}, 
$$
\dot u(t)\cosh(t)^2\le \dot u(0+)= \dot y(0+)
$$
for $t\in (0,t_1).$  Multiplying through by $\text{sech}(t)^2$ and integrating from $0$ to $t$ gives 
\be\label{initialEstimateonU}
u(t)\le y(0)+\dot y(0+)\tanh(t)\quad t\in[0,t_1].
\ee
That is, 
\be\label{initialEstimateonY}
y(t)=\cosh(t)u(t)\le 
 \cosh(t)y(0)+\sinh(t)\dot y(0+)
\ee
for  $t\in[0,t_1]$. 

\par By \eqref{slopeDecreaseY} and \eqref{UdotLessThan}, we likewise have
$$
\dot u(t)\cosh(t)^2\le \dot u(t_1+)\cosh(t_1)^2\le \dot u(t_1-)\cosh(t_1)^2\le u(0+)= \dot y(0+)
$$
for $t\in (t_1,t_2)$. Again we multiply through by $\text{sech}(t)^2$ and integrate from $t_1$ to $t\in (t_1,t_2)$ to get  
\begin{align*}
u(t)&\le u(t_1)+ \dot y(0+)(\tanh(t)-\tanh(t_1))\\
&\le y(0)+\dot y(0+)\tanh(t_1)+\dot y(0+)(\tanh(t)-\tanh(t_1))\\
&\le y(0)+\dot y(0+)\tanh(t).
\end{align*}
In particular, \eqref{initialEstimateonY} holds for $t\in [t_1,t_2]$.  We can argue similarly to show that 
\eqref{initialEstimateonY} also holds on the intervals $[t_2,t_3], \dots, [t_n,T)$ .
\end{proof}

This leads to a stability estimate. 
\begin{prop}\label{StabilityProp}
Suppose $i,j\in\{1,\dots, N\}$, $x_i\ge x_j$ and $t\ge 0$. Then 
\be\label{timeZeroEst}
\gamma_i(t)-\gamma_j(t)\le \cosh(\sqrt{c}t)(x_i-x_j)+\frac{1}{\sqrt{c}}\sinh(\sqrt{c}t)\int^{x_i}_{x_j}|v_0'(x)|dx.
\ee
\end{prop}
\begin{proof}
Without loss of generality, we may assume $x_1\le\dots\le x_N$ so that the sticky particle trajectories are ordered $\gamma_1\le \dots\le\gamma_N$. Under this assumption, it suffices to verify 
\be\label{timeZeroEstiplusone}
\gamma_{i+1}(t)-\gamma_i(t)\le \cosh(\sqrt{c}t)(x_{i+1}-x_i)+\frac{1}{\sqrt{c}}\sinh(\sqrt{c}t)|v_{i+1}-v_i|
\ee
for $t\ge 0$.  For if $j,k\in \{1,\dots, N\}$ with $k>j$, 
\begin{align}
\gamma_k(t)-\gamma_j(t)&= \sum^{k-1}_{i=j}(\gamma_{i+1}(t)-\gamma_i(t))\\
&\le \sum^{k-1}_{i=j}\left(\cosh(\sqrt{c}t)(x_{i+1}-x_i)+\frac{1}{\sqrt{c}}\sinh(\sqrt{c}t)
|v_{i+1}-v_i|\right)\\
&= \cosh(\sqrt{c}t)(x_k-x_j)+\frac{1}{\sqrt{c}}\sinh(\sqrt{c}t)\sum^{k-1}_{i=j}|v_{i+1}-v_i|\\
&\le  \cosh(\sqrt{c}t)(x_k-x_j)+\frac{1}{\sqrt{c}}\sinh(\sqrt{c}t)\sum^{k-1}_{i=j}\int^{x_{i+1}}_{x_i}|v_0'(x)|dx\\
&= \cosh(\sqrt{c}t)(x_k-x_j)+\frac{1}{\sqrt{c}}\sinh(\sqrt{c}t)\int^{x_{k}}_{x_j}|v_0'(x)|dx.
\end{align}

\par To this end, we fix $i\in\{1,\dots, N\}$ and set
$$
T:=\inf\{t\ge 0: \gamma_{i+1}(t)-\gamma_i(t)=0\}.
$$
In order to verify \eqref{timeZeroEstiplusone}, it is enough to show
\be\label{LasttimeZeroEst}
\gamma_{i+1}(t)-\gamma_i(t)\le \cosh(\sqrt{c}t)(x_{i+1}-x_i)+\frac{1}{\sqrt{c}}\sinh(\sqrt{c}t)|v_{i+1}-v_i|,\quad t\in [0,T].
\ee
We will do so by applying the previous lemma to the restriction of the function
$$
y(t):=\gamma_{i+1}(t)-\gamma_i(t), \quad t\ge 0
$$
to $[0,T)$. In particular, we note that $y(t)=0$ for $t\ge T$ whenever $T$ is finite.

\par We first claim
\be\label{DotGmamaiplusone}
\dot\gamma_{i+1}(s+)\le \dot\gamma_{i+1}(s-).
\ee
Note that if $\gamma_{i+1}$ does not have a first intersection time at $s\in (0,T)$, then $\gamma_{i+1}$ is $C^1$ near $s$ and so 
$$
\dot\gamma_{i+1}(s)=\dot\gamma_{i+1}(s+)=\dot\gamma_{i+1}(s-).
$$
Alternatively let us suppose $\gamma_{i+1}$ has a first intersection time at $s$. As a result, there are trajectories $\gamma_{i+2}, \dots, \gamma_{i+r}$  $(r\ge 2)$ such that 
$$
\gamma_{i+1}(s)=\gamma_{i+2}(s)=\dots= \gamma_{i+r}(s)
$$
and 
\be\label{AveragingForQSPS}
\dot\gamma_{i+j}(s+)=\frac{m_{i+1}\dot\gamma_{i+1}(s-)+\dots + m_{i+r}\dot\gamma_{i+r}(s-)}{m_{i+1}+\dots+m_{i+r}}
\ee
$j=1,\dots, r$. 

\par Also note that as $\gamma_{i+1}\le \gamma_{i+j}$ for $j=2,\dots, r$,
$$
\frac{\gamma_{i+1}(s+h)-\gamma_{i+1}(s)}{h}\ge \frac{\gamma_{i+j}(s+h)-\gamma_{i+j}(s)}{h}
$$
for all $h<0$ small. By Remark \ref{interchangeRemark}, we can send $h\rightarrow 0^-$ and conclude
$$
\dot\gamma_{i+1}(s-)\ge \dot\gamma_{i+j}(s-).
$$
It then follows from \eqref{AveragingForQSPS} (with $j=1$) that 
 $$
\dot \gamma_{i+1}(s+)\le\frac{m_{i+1}\dot\gamma_{i+1}(s-)+\dots + m_{i+r}\dot\gamma_{i+1}(s-)}{m_{i+1}+\dots+m_{i+r}}=\dot\gamma_{i+1}(s-),
 $$
  which is \eqref{DotGmamaiplusone}. A similar argument gives 
\be\label{DotGmamajusti}
\dot\gamma_{i}(s+)\ge \dot\gamma_{i}(s-)
\ee
for each $s\in (0,T)$. Combining \eqref{DotGmamaiplusone} and \eqref{DotGmamajusti} 
 \be\label{dotYnotgoingup}
 \dot y(s+)=\dot\gamma_{i+1}(s+)-\dot\gamma_{i}(s+)\le \dot\gamma_{i+1}(s-)-\dot\gamma_{i}(s-)= \dot y(s-)
 \ee
for all $s\in (0,T)$.

\par As $x\mapsto W'(x)+cx$ is nondecreasing, 
\begin{align*}
\ddot y(t)&=\ddot\gamma_i(t)-\ddot\gamma_j(t)\\
&=-\sum_{k=1}^Nm_k\left(W'(\gamma_i(t)-\gamma_k(t))-W'(\gamma_j(t)-\gamma_k(t))\right) \\
&\le \sum_{k=1}^Nm_k c\left(\gamma_i(t)-\gamma_j(t)\right)\\
&= c(\gamma_i(t)-\gamma_j(t))\\
&=y(t)
\end{align*}
for all but finitely many $t>0$.  Therefore, \eqref{LasttimeZeroEst} follows from Lemma \ref{ODEineqLemma2}.  
\end{proof}

\subsection{Associated trajectory map}
We are now ready to show how to design a solution of \eqref{FlowMapEqn} with $\rho_0$ given by \eqref{discreterhozero}.  For $t\ge0$, we define 
$$
X(t): \{x_1,\dots, x_N\}\rightarrow \R; x_i\mapsto \gamma_i(t).
$$
We will also write
$$
X(x_i,t)=\gamma_i(t)
$$
for $ i=1,\dots, N$ and $t\ge 0$.  The following proposition details all the important features of $X$.

\begin{prop}\label{DiscreteSolnProp}
The mapping $X$ has the following properties. 

\begin{enumerate}[(i)]

\item $X(0)=\textup{id}_\R$ and 
\be
\dot X(t)=\E_{\rho_0}\left[v_0-\displaystyle\int^t_0(W'*\rho_s)(X(s))ds\bigg| X(t)\right]
\ee
for all but finitely many $t\ge 0$. Both equalities hold on the support of $\rho_0$. 

\item $E(t)\le E(s)$, for $s\le t$. Here 
\begin{align*}
E(\tau):=\int_{\R}\frac{1}{2}\dot X(\tau+)^2d\rho_0+\int_{\R}\int_{\R}\frac{1}{2}W(X(y,\tau)-X(z,\tau))d\rho_0(y)d\rho_0(z).
\end{align*}

\item $X: [0,\infty)\rightarrow L^2(\rho_0); t\mapsto X(t)$ is locally Lipschitz continuous.

\item For $t\ge 0$ and $y,z\in\textup{supp}(\rho_0)$ with $y\le z$, 
$$
0\le X(z,t)-X(y,t)\le \cosh(\sqrt{c}t)(z-y)+\frac{1}{\sqrt{c}}\sinh(\sqrt{c}t)\int^z_y|v_0'(x)|dx.
$$

\item For each $0<s\le t$ and $y,z\in\textup{supp}(\rho_0)$
$$
\frac{|X(y,t)-X(z,t)|}{\sinh(\sqrt{c}t)}\le \frac{|X(y,s)-X(z,s)|}{\sinh(\sqrt{c}s)}.
$$

\item For each $0<s\le t$, there is a function $f_{t,s}:\R\rightarrow \R$ which satisfies the Lipschitz condition \eqref{fteessLip} 
and
$$
X(y,t)=f_{t,s}(X(y,s))
$$
for $y\in\textup{supp}(\rho_0)$.
\end{enumerate}

\end{prop}
\begin{proof}
\underline{Part $(i)$}: As $X(x_i,0)=x_i$, 
$$
X(0)=\id
$$
on $\text{supp}(\rho_0)$. Also note that if $g:\R\rightarrow \R$ and $t\ge 0$, then Corollary \ref{AveragingProp} gives
\begin{align*}
\int_{\R}g(X(t))\dot X(t+)d\rho_0&=\sum^N_{i=1}m_ig(\gamma_i(t))\dot\gamma_i(t+)\\
&= \sum^N_{i=1}m_ig(\gamma_i(t))\left[\dot\gamma_i(0+)-\int^t_0\left(\sum^N_{j=1}m_jW'(\gamma_i(\tau)-\gamma_j(\tau))\right)d\tau \right]\\
&=\int_{\R}g(X(t))\left[v_0-\displaystyle\int^t_0(W'*\rho_\tau)(X(\tau))d\tau\right]d\rho_0.
\end{align*}
In particular,  
$$
\int_{\R}g(X(t))\dot X(t)d\rho_0=\int_{\R}g(X(t))\left[v_0-\displaystyle\int^t_0(W'*\rho_\tau)(X(\tau))d\tau\right]d\rho_0
$$
for all but finitely many $t\ge 0$. 

\par Define 
\be
\begin{cases}
v(x,t)=\dot\gamma_i(t+), \quad &x=\gamma_i(t)\\
\hspace{.48in}=0, \quad & \text{otherwise}.
\end{cases}
\ee
By parts $(iii)$ and $(iv)$ of Proposition \ref{StickyParticlesExist}, $v$ is well defined. Moreover, it is routine to check that $v:\R\times[0,\infty)\rightarrow \R$ is Borel measurable. Furthermore,
$$
\dot X(t)=v(X(t),t)
$$
on the support of $\rho_0$ for all but finitely many $t\ge 0$. It follows that $X$ satisfies the pressureless Euler flow equation \eqref{FlowMapEqn} for all but finitely many $t\ge 0$.

\par \underline{Part $(ii)$}: In view of \eqref{ContNonincreaseEnergy},
\begin{align*}
E(t)&=\int_{\R}\frac{1}{2}\dot X(t+)^2d\rho_0+\int_{\R}\int_{\R}\frac{1}{2}W(X(y,t)-X(z,t))d\rho_0(y)d\rho_0(z)\\
&=\frac{1}{2}\sum^N_{i=1}m_i\dot\gamma_i(t+)^2+\frac{1}{2}\sum^N_{i,j=1}m_im_jW(\gamma_i(t)-\gamma_j(t))\\
&\le \frac{1}{2}\sum^N_{i=1}m_i\dot\gamma_i(s+)^2+\frac{1}{2}\sum^N_{i,j=1}m_im_jW(\gamma_i(s)-\gamma_j(s))\\
&=\int_{\R}\frac{1}{2}\dot X(t+)^2d\rho_0+\int_{\R}\int_{\R}\frac{1}{2}W(X(y,t)-X(z,t))d\rho_0(y)d\rho_0(z)\\
&=E(s).
\end{align*}

\par \underline{Part $(iii)$}: By the energy estimate \eqref{DiscreteEnergyEst2}, 
\begin{align}\label{LipEst}
\int_{\R}(X(t)-X(s))^2d\rho_0&\le (t-s)\int^t_s\int_{\R}\dot X(\tau)^2d\rho_0 d\tau\\
&= (t-s)\int^t_s\sum^N_{i=1}m_i\dot\gamma_i(\tau)^2d\tau\\
&\le (t-s)(\vartheta(t)-\vartheta(s))\left(\sum^N_{i=1}m_iv_i^2+\frac{1}{2}\sum^N_{i,j=1}m_im_jW'(x_i-x_j)^2 \right)\\
&\le (t-s)(\vartheta(t)-\vartheta(s))\left(\int_{\R}v_0^2d\rho_0+\frac{1}{2}\int_{\R}\int_{\R}W'(x-y)^2d\rho_0(x)d\rho_0(y)\right).
\end{align}
Since $\vartheta$ is smooth, $X:[0,\infty)\rightarrow L^2(\rho_0)$ is locally Lipschitz continuous. 

\par \underline{Part $(iv)$, Part $(v)$ and $(vi)$}: Part $(iv)$ follows from Proposition \ref{StabilityProp}, Part $(v)$ is a corollary of Proposition \ref{PropQSPP}, and part $(vi)$ is due to Corollary \ref{measurabilityProp}. 
\end{proof}

\par Observe that as $v_0:\R\rightarrow \R$ is absolutely continuous,  
\be\label{modulusOmega}
\omega(r):=\sup\left\{\int^b_a|v_0'(x)|dx \;:\; 0\le b-a\le r\right\}
\ee
tends to $0$ as $r\rightarrow 0^+$ and is sublinear.  By part $(iv)$ of the above proposition, 
\be\label{XdiscreteUnifCont}
|X(y,t)-X(z,t)|\le \cosh(\sqrt{c}t)|y-z|+\frac{1}{\sqrt{c}}\sinh(\sqrt{c}t)\omega(|y-z|)
\ee
for $y,z$ belonging to the support of $\rho_0$. As a result, $X(t)$ is uniformly continuous on the support of $\rho_0$.  In particular, we may extend $X(t)$ to a uniformly continuous function on $\R$ which satisfies \eqref{XdiscreteUnifCont} and agrees with $X(t)$ on the support of $\rho_0$.  Without any loss of generality, we will identify $X(t)$ with this extension and assume $X(t)$ is uniformly continuous on $\R$.

\section{Existence of solutions}\label{CompactSect}
Now let $\rho_0\in {\cal P}(\R)$ with
\be
\int_{\R}x^2d\rho_0(x)<\infty.
\ee
We can select a sequence $(\rho_0^k)_{k\in \N}\subset {\cal P}(\R)$ such that each $\rho^k_0$ is a convex combination of Dirac measures, $\rho_0^k\rightarrow \rho_0$ narrowly, and
\be\label{rhoQuadraticConv}
\lim_{k\rightarrow\infty}\int_{\R}x^2d\rho^k_0(x)=\int_{\R}x^2d\rho_0(x).
\ee
We recommend, for instance, the reference \cite{Bolley} for a discussion on how to design such a sequence.  

\par In view of Proposition \ref{DiscreteSolnProp}, there is a locally Lipschitz continuous mapping $X^k:[0,\infty)\rightarrow L^2(\rho^k_0)$ which satisfies
\be
\begin{cases}
\dot X^k(t)=\E_{\rho_0^k}\left[v_0-\displaystyle\int^t_0(W'*\rho^k_s)(X^k(s))ds\bigg| X^k(t)\right], \quad a.e.\;t\ge 0\\
X^k(0)=\id
\end{cases}
\ee
for each $k\in \N$.  Here 
$$
\rho^k_t:=X^k(t)_{\#}\rho^k_0, \quad t\ge 0.
$$
In this section, we will show that there is a subsequence of $(X^k)_{k\in \N}$ which converges in an appropriate sense to a solution of the pressureless Euler flow equation \eqref{FlowMapEqn} and satisfies \eqref{xInit}.  Then we will show how to use this solution to generate a corresponding weak solution of the pressureless Euler equations.

\subsection{Compactness}
We will prove that $(X^{k})_{k\in \N}$ has a subsequence which converges in a strong sense.  The limit mapping will be our candidate for a solution of the pressureless Euler flow equation.

\begin{prop}\label{StrongCompactXkayjay}
There is a subsequence $(X^{k_j})_{j\in \N}$ and a locally Lipschitz $X: [0,\infty)\rightarrow L^2(\rho_0)$ such that 
\be\label{StrongConvLimXkayJay}
\lim_{j\rightarrow\infty}\int_{\R}h(\textup{id}_\R,X^{k_j}(t))d\rho^{k_j}_0=\int_{\R}h(\textup{id}_\R,X(t))d\rho_0
\ee
for each $t\ge 0$ and continuous $h:\R^2\rightarrow \R$ with
\be
\sup_{(x,y)\in \R^2}\frac{|h(x,y)|}{1+x^2+y^2}<\infty.
\ee
Furthermore, $X$ has the following properties.  

\begin{enumerate}[(i)]

\item For $t\ge 0$ and $y,z\in\textup{supp}(\rho_0)$ with $y\le z$, 
$$
0\le X(z,t)-X(y,t)\le \cosh(\sqrt{c}t)(z-y)+\frac{1}{\sqrt{c}}\sinh(\sqrt{c}t)\int^z_y|v_0'(x)|dx.
$$

\item For each $0<s\le t$ and $y,z\in\textup{supp}(\rho_0)$
$$
\frac{|X(y,t)-X(z,t)|}{\sinh(\sqrt{c}t)}\le \frac{|X(y,s)-X(z,s)|}{\sinh(\sqrt{c}s)}.
$$

\item For each $0<s\le t$, there is a function $f_{t,s}:\R\rightarrow \R$ which satisfies the Lipschitz condition \eqref{fteessLip} 
and
$$
X(y,t)=f_{t,s}(X(y,s))
$$
for $y\in\textup{supp}(\rho_0)$.
\end{enumerate}

\end{prop}

\begin{proof}
\underline{1. Weak convergence}: Define
\be
\sigma^k_t:=(\id,X^{k}(t))_{\#}\rho^k_0
\ee
for $t\ge 0$ and $k\in \N$. By \eqref{LipEst}, 
\begin{align*}
\left(\int_{\R}(X^k(t))^2d\rho^k_0\right)^{1/2}&\le\left(\int_{\R}(X^k(t)-X^k(0))^2d\rho_0^k\right)^{1/2}+\left(\int_{\R}(X^k(0))^2d\rho^k_0\right)^{1/2}\\
&\le \sqrt{t\vartheta(t)}\left(\int_{\R}v_0^2d\rho^k_0+\frac{1}{2}\int_{\R}\int_{\R}W'(x-y)^2d\rho^k_0(x)d\rho^k_0(y)\right)^{1/2}\\
&\hspace{1in}+\left(\int_{\R}x^2d\rho^k_0(x)\right)^{1/2}.
\end{align*}
And as $v_0$ and $W'$ grow at most linearly, there are constants $A, B\ge 0$
$$
\left(\int_{\R}(X^k(t))^2d\rho^k_0\right)^{1/2}\le A\sqrt{t\vartheta(t)} +B
$$
for each $t\ge 0$ and $k\in \N$. 

\par It follows that 
\be\label{SecondMomentEst}
\sup_{k\in\N}\iint_{\R^2}(x^2+y^2)d\sigma^k_t(x,y)<\infty
\ee
for each $t\ge 0.$ As a result, $(\sigma^k_t)_{k\in \N}$ is narrowly precompact.  Moreover, for Lipschitz continuous $h:\R^2\rightarrow \R$,
\begin{align*}
\iint_{\R^2}h(x,y)d\sigma^k_t(x,y)-\iint_{\R^2}h(x,y)d\sigma^k_s(x,y)&=\int_{\R}h(\id,X^k(t))d\rho^k_0-\int_{\R}h(\id,X^k(s))d\rho^k_0\\
&\le \text{Lip}(h) \int_{\R}|X^k(t)-X^k(s)| d\rho^k_0\\
&\le \text{Lip}(h) \int^t_s\left(\int_{\R}|\dot X^k(\tau)| d\rho^k_0\right)d\tau\\
&\le \text{Lip}(h) \int^t_s (A\sqrt{\tau\vartheta(\tau)} +B)d\tau.
\end{align*} 
In terms of
the metric \eqref{NarrowMetric}, we then have 
\be
\mathcal{d}(\sigma^k_t,\sigma^k_s)\le  \int^t_s (A\sqrt{\tau\vartheta(\tau)} +B)d\tau
\ee
for $k\in \N$ and $0\le s\le t$. 

\par By the Arzel\`a-Ascoli, there is a subsequence $(\sigma^{k_j})_{j\in \N}$ and a narrowly continuous path $\sigma:[0,\infty)\rightarrow {\cal P}(\R);t\mapsto \sigma_t$ such that  $\sigma^{k_j}_t\rightarrow \sigma_t$ uniformly for $t$ belonging to compact subsets of $[0,\infty)$.  We can also use this narrow convergence and \eqref{SecondMomentEst} to show  
\be
\lim_{j\rightarrow\infty}\iint_{\R^2}|y|d\sigma^{k_j}_t(x,y)=\iint_{\R^2}|y|d\sigma_t(x,y)
\ee
for each $t\ge 0$. Further, 
\be
\int_{\R}\phi d\rho_0 =\lim_{j\rightarrow\infty}\int_{\R}\phi d\rho^{k_j}_0 =\lim_{j\rightarrow\infty}\iint_{\R^2}\phi(x)d\sigma^{k_j}_t(x,y)=\iint_{\R^2}\phi(x)d\sigma_t(x,y)
\ee
for $\phi\in C_b(\R)$.  

\par By the disintegration theorem (Theorem 5.3.1 of \cite{AGS}), there is a family of probability measures $(\zeta^x_t)_{x\in \R}$ for which 
\be
\iint_{\R^2}h(x,y)d\sigma_t(x,y)=\int_{\R}\left(\int_{\R}h(x,y)d\zeta^x_t(y)\right)d\rho_0(x).
\ee
Set 
$$
X(x,t):=\int_{\R}yd\zeta^x_t(y), \quad (x,t)\in \R\times[0,\infty),
$$
and as before we will write $X(t):\R\rightarrow \R; x\mapsto X(x,t)$. Observe
\begin{align}\label{WeakConvXkayJay}
\lim_{j\rightarrow\infty}\int_{\R}\phi X^{k_j}(t)d\rho^{k_j}_0&=\lim_{j\rightarrow\infty}\int_{\R^2 }\phi(x) yd\sigma^{k_j}_t(x,y)\nonumber \\
&=\int_{\R^2 }\phi(x) yd\sigma_t(x,y)\nonumber\\
&=\int_{\R}\left(\int_{\R}\phi(x) yd\zeta^x_t(y)\right)d\rho_0(x)\nonumber\\
&=\int_{\R}\phi(x)\left(\int_{\R}yd\zeta^x_t(y)\right)d\rho_0(x)\nonumber\\
&=\int_{\R}\phi X(t)d\rho_0
\end{align}
for each $t\ge 0$.

\par \underline{2. Strong convergence}:  By \eqref{XdiscreteUnifCont},
\be
|X^k(y,t)-X^k(z,t)|\le \cosh(\sqrt{c}t)|y-z|+\frac{1}{\sqrt{c}}\sinh(\sqrt{c}t)\omega(|y-z|)
\ee
for $y,z\in \R$. As $\lim_{r\rightarrow 0^+}\omega(r)=0$, $(X^k(t))_{k\in \N}$ is uniformly equicontinuous. Moreover, 
\be
|X^k(y,t)|\le |X^k(y,t)-X^k(z,t)|+|X^k(z,t)|
\ee
and integrating over $z\in \R$ gives
\begin{align*}
|X^k(y,t)|&\le \int_{\R}|X^k(y,t)-X^k(z,t)|d\rho^k_0(z)+\int_{\R}|X^k(z,t)|d\rho^k_0(z) \\
&\le \cosh(\sqrt{c}t)\int_{\R}|y-z|d\rho^k_0(z)+\frac{1}{\sqrt{c}}\sinh(\sqrt{c}t) \int_{\R}\omega(|y-z|)d\rho^k_0(z)\\
&\quad +\int_{\R}|X^k(z,t)|d\rho^k_0(z) \\
&\le \cosh(\sqrt{c}t)\int_{\R}|y-z|d\rho^k_0(z)+\frac{1}{\sqrt{c}}\sinh(\sqrt{c}t) \int_{\R}\omega(|y-z|)d\rho^k_0(z)\\
&\quad +A\sqrt{t\vartheta(t)} +B.
\end{align*}
By \eqref{rhoQuadraticConv} and the at most sublinear growth of $\omega$, there are constants $a,b\ge 1$ such that 
\be\label{XkayUpperBound}
|X^k(y,t)|\le a e^{b t^2}(|y|+1)
\ee
for $k\in \N$, $y\in \R$ and $t\ge0$.

\par It follows that a subsequence of $(X^{k_j}(t))_{j\in \N}$ (which we will not relabel) converges locally uniformly on $\R$ to a continuous function $Y:\R\rightarrow \R$. It is easy to check
\be
\lim_{j\rightarrow\infty}\int_{\R}\phi X^{k_j}(t)d\rho^{k_j}_0=\int_{\R}\phi Yd\rho_0
\ee
for each $\phi\in C_b(\R)$. So if $(X^{k_j}(t))_{j\in \N}$ has another subsequence which converges locally uniformly on $\R$ to $Z$, then 
\be
\int_{\R}\phi Yd\rho_0=\int_{\R}\phi Zd\rho_0.
\ee
In particular, $Y=Z$ $\rho_0$ almost everywhere. Since $Y$ and $Z$ are continuous, it is routine to check that $Y=Z$ on the entire support of $\rho_0$.  

\par By \eqref{WeakConvXkayJay}, we also have that $X(t)=Y$ $\rho_0$ almost everywhere. Without loss of generality, we can redefine $X(t)=Y$ to ensure that $X^{k_j}(t)\rightarrow X(t)$ locally uniformly on $\R$. By estimate \eqref{XkayUpperBound} and Proposition \ref{LemmaVariantNarrowCon}, we also have 
\be
\lim_{j\rightarrow\infty}\int_{\R}X^{k_j}(t)^2d\rho^{k_j}_0=\int_{\R}X(t)^2d\rho_0.
\ee
In particular, 
\be
\lim_{j\rightarrow\infty}\iint_{\R^2}(x^2+y^2)d\sigma^{k_j}_t(x,y)=\iint_{\R^2}(x^2+y^2)d\sigma_t(x,y),
\ee
which when combined the narrow converge  $\sigma^{k_j}_t\rightarrow \sigma_t$ in ${\cal P}(\R^2)$ gives \eqref{StrongConvLimXkayJay}. 

\par In view of \eqref{LipEst},  we also have
\begin{align*}
\int_{\R}(X(t)-X(s))^2d\rho_0&=\lim_{j\rightarrow\infty}\int_{\R}(X^{k_j}(t)-X^{k_j}(s))^2d\rho^{k_j}_0\\
&\le (t-s)(\vartheta(t)-\vartheta(s))\lim_{j\rightarrow\infty}\left(\int_{\R}v_0^2d\rho^{k_j}_0+\frac{1}{2}\int_{\R}\int_{\R}W'(x-y)^2d\rho^{k_j}_0(x)d\rho^{k_j}_0(y)\right)\\
&\le (t-s)(\vartheta(t)-\vartheta(s))\left(\int_{\R}v_0^2d\rho_0+\frac{1}{2}\int_{\R}\int_{\R}W'(x-y)^2d\rho_0(x)d\rho_0(y)\right).
\end{align*} 
Here we used that $v_0$ and $W'$ are continuous and grow at most linearly.  Therefore, the mapping $X:[0,\infty)\rightarrow L^2(\rho_0)$ is locally Lipschitz continuous. 

\par \underline{3. Properties of the limit}: Suppose $y,z\in\text{supp}(\rho_0)$ with $y<z$.  As $\rho^{k_j}\rightarrow \rho_0$ narrowly in ${\cal P}(\R)$, 
there are sequences $y^j,z^j\in\text{supp}(\rho^{k_j}_0)$ such that $y^j\rightarrow y$ and $z^j\rightarrow z$ (Proposition 5.1.8 in \cite{AGS}).  Without any loss of generality, we may 
suppose that $y^j<z^j$ for all $j\in \N$ as this occurs for all $j$ large enough. By part $(iv)$ of Proposition \ref{DiscreteSolnProp}, 
\be
0\le X^{k_j}(z^j,t)-X^{k_j}(y^j,t)\le \cosh(\sqrt{c}t)(z^j-y^j)+\frac{1}{\sqrt{c}}\sinh(\sqrt{c}t)\int^{z^j}_{y^j}|v_0'(x)|dx.
\ee
Since $X^{k^j}(t)\rightarrow X(t)$ locally uniformly, we can send $j\rightarrow\infty$ in order to conclude property $(i)$ of this theorem. Property $(ii)$ can be proved similarly. 

\par Now let $0<s\le t$ and $y\in\text{supp}(\rho_0)$. As above, we may select $y^j\in\text{supp}(\rho^{k_j}_0)$ such that $y^j\rightarrow y$. Appealing to part $(vi)$ of Proposition \ref{DiscreteSolnProp}, there is a sequence of functions $(f^{k_j}_{t,s})_{j\in \N}$ in which each element satisfies the Lipschitz condition \eqref{fteessLip} and
\be\label{XkayJayFinalest}
X^{k_j}(y^j,t)=f^{k_j}_{t,s}(X^{k_j}(y^j,s))
\ee
for each $j\in \N$. In particular, 
\begin{align*}
|f^{k_j}_{t,s}(x)|&\le |f^{k_j}_{t,s}(x)-f^{k_j}_{t,s}(X^{k_j}(y^j,s))|+|f^{k_j}_{t,s}(X^{k_j}(y^j,s))|\\
&\le \frac{\sinh(\sqrt{c}t)}{\sinh(\sqrt{c}s)}|x-X^{k_j}(y^j,s)|+|X^{k_j}(y^j,t)|.
\end{align*}
As $X^{k^j}(t)\rightarrow X(t)$ and $X^{k^j}(s)\rightarrow X(s)$ locally uniformly, $|f^{k_j}_{t,s}(x)|$ is bounded above for $x$ belonging to compact subsets of $\R$ independently of $j\in \N$. As a result, $(f^{k_j}_{t,s})_{j\in \N}$ has a locally uniformly convergent subsequence. Therefore, we can send $j\rightarrow\infty$ along an appropriate subsequence in \eqref{XkayJayFinalest} to obtain part $(iii)$.
\end{proof}
For the remainder of this section, let $X$ denote the mapping obtained in Proposition \ref{StrongCompactXkayjay}. 
\begin{cor}\label{dotXfunctionofX}
For Lebesgue almost every $t>0$, there is a Borel function $u:\R\rightarrow \R$ for which 
\be
\dot X(t)=u(X(t))
\ee
$\rho_0$ almost everywhere. 
\end{cor}
\begin{proof}
Let $t$ be a time such that
\be\label{dotXlimit}
\dot X(t)=\lim_{n\rightarrow \infty}n\left(X(t+1/n)-X(t)\right)
\ee
in $L^2(\rho_0)$. We recall that since $X:[0,\infty)\rightarrow L^2(\rho_0)$ is locally Lipschitz, the set of all such $t$ has full measure in $[0,\infty)$. Moreover, we may suppose that the limit \eqref{dotXlimit} holds $\rho_0$ almost everywhere in $\R$ as it does for a subsequence.  By Theorem 1.19 in \cite{Folland}, we may further assume that the limit \eqref{dotXlimit} holds everywhere on some Borel $S\subset \R$ with $\rho_0(S)=1$.  

\par In view of part $(iii)$ of Proposition \ref{StrongCompactXkayjay}, 
$$
\dot X(t)=\lim_{n\rightarrow \infty}u_n\left(X(t)\right)
$$ 
on $S$. Here
$$
u_n:=n(f_{t+1/n,t}-\id)
$$
is a Borel measurable function for each $n\in \N$.  Consequently, $u_n\left(X(t)\right)|_S$ is measurable with respect to the Borel sigma-sub-algebra generated by $X(t)|_S$
$$
{\cal F}:=\{\{y\in S: X(y,t)\in B\}: B\subset \R\;\text{Borel}\}.
$$
As $\dot X(t)|_S$ is a pointwise limit of ${\cal F}$ measurable functions, it is ${\cal F}$ measurable itself (Proposition 2.7 \cite{Folland}). It follows that there is a Borel function $u:\R \rightarrow \R$ such that 
$$
\dot X(t)|_S=u(X(t))|_S.
$$  
That is, $\dot X(t)=u(X(t))$ $\rho_0$ almost everywhere. 
\end{proof}

\begin{proof}[Proof of Theorem \ref{mainThm}] 1. \underline{Initial condition}: The limit \eqref{StrongConvLimXkayJay} taken when $t=0$ implies $X(0)=\id$ $\rho_0$ almost everywhere. 

\par 2.  \underline{Flow equation}: We next claim 
\be\label{SolnGoal}
\int^t_s\int_{\R}\dot X(\tau) h(X(\tau))d\rho_0d\tau=\int^t_s\int_{\R}\left[v_0-\displaystyle\int^\tau_0(W'*\rho_\xi)(X(\xi))d\xi\right]h(X(t))d\rho_0d\tau
\ee
for each continuous $h:\R\rightarrow\R$ which satisfies 
\be
\sup_{x\in \R}\frac{|h(x)|}{1+|x|}<\infty
\ee
and each $0\le s\le t$.  In view of Corollary \ref{dotXfunctionofX}, this would imply that $X$ is a solution of the pressureless Euler flow equation.
To this end, we set $F(y):=\int^y_0 h(x)dx$ for $y\in \R$ and note that $F$ grows at most quadratically.  By Proposition \ref{StrongCompactXkayjay},
\begin{align*}
\lim_{j\rightarrow\infty}\int^t_s\int_{\R}\dot X^{k_j}(\tau) h(X^{k_j}(\tau))d\rho^{k_j}_0d\tau&=\lim_{j\rightarrow\infty}\int_{\R}\left(\int^t_s\dot X^{k_j}(\tau) h(X^{k_j}(\tau))d\tau\right)d\rho^{k_j}_0\\
&=\lim_{j\rightarrow\infty}\int_{\R}\left(\int^t_s\frac{d}{d\tau}F(X^{k_j}(\tau))d\tau\right)d\rho^{k_j}_0\\
&=\lim_{j\rightarrow\infty}\int_{\R}\left(F(X^{k_j}(t))-F(X^{k_j}(s))\right)d\rho^{k_j}_0\\
&=\int_{\R}\left(F(X(t))-F(X(s))\right)d\rho_0.
\end{align*}
Consequently,
\be\label{WeakConvDotXkayjay}
\lim_{j\rightarrow\infty}\int^t_s\int_{\R}\dot X^{k_j}(\tau) h(X^{k_j}(\tau))d\rho^{k_j}_0d\tau=\int^t_s\int_{\R}\dot X(\tau) h(X(\tau))d\rho_0d\tau.
\ee

\par Let us fix $\tau\ge 0$ for the moment and consider the integral
\begin{align*}
&\int_{\R}\left[v_0-\displaystyle\int^\tau_0(W'*\rho^{k_j}_\xi)(X^{k_j}(\xi))d\xi\right] h(X^{k_j}(\tau))d\rho^{k_j}_0\\
&\quad = \int_{\R}v_0 h(X^{k_j}(\tau))d\rho^{k_j}_0-\int^\tau_0\left[\int_{\R}\int_{\R}W'(X^{k_j}(y,\xi)-X^{k_j}(z,\xi)) h(X^{k_j}(y,\tau)) d\rho^{k_j}_0(z)d\rho^{k_j}_0(y)\right]d\xi.
\end{align*}
In view of Proposition \ref{StrongCompactXkayjay} and the at most linear growth of $v_0$, 
\be
\lim_{j\rightarrow\infty}\int_{\R}v_0 h(X^{k_j}(\tau))d\rho^{k_j}_0=\int_{\R}v_0 h(X(\tau))d\rho_0
\ee
as $j\rightarrow\infty$. Also observe
\be
\lim_{j\rightarrow\infty}W'(X^{k_j}(y,\xi)-X^{k_j}(z,\xi)) h(X^{k_j}(y,t))=W'(X(y,\xi)-X(z,\xi)) h(X(y,\tau))
\ee
locally uniformly for $(y,z)\in \R^2$ and each fixed $\xi\in [0,\tau]$. 

\par Choosing $C$ so large that
$$
|h(x)|+|W'(x)|\le C(1+|x|)\quad (x\in \R)
$$
gives
\begin{align}\label{SimpleConvBoundWprime}
&|W'(X^{k_j}(y,\xi)-X^{k_j}(z,\xi)) h(X^{k_j}(y,\tau))|  \nonumber \\ 
&\quad\quad \le C(1+|X^{k_j}(y,\xi)|+ |X^{k_j}(z,\xi)|)\cdot C(1+|X^{k_j}(y,\tau)|)\nonumber  \\
&\quad\quad \le C(1+a e^{b \xi^2}(|y|+1)+ a e^{b \xi^2}(|z|+1))\cdot C(1+a e^{b \tau^2}(|y|+1))\nonumber  \\
&\quad\quad \le C^2(1+a e^{b \tau^2}(|y|+1)+ a e^{b \tau^2}(|z|+1))^2 \nonumber\\
&\quad\quad \le (Ca e^{b \tau^2})^2(3+|y|+|z|)^2 \nonumber \\
&\quad\quad \le (Ca e^{b \tau^2})^22(9+y^2+z^2)\nonumber \\
&\quad\quad \le 18(Ca e^{b \tau^2})^2(1+y^2+z^2)
\end{align}
for $j\in \N$ and $\xi\in [0,\tau]$. Here $a, b$ are the constants from inequality \eqref{XkayUpperBound}.

\par We can then appeal to Proposition \ref{LemmaVariantNarrowCon}  to conclude 
\begin{align}\label{FirstThingWeNeedforKineticEnergyLimit}
&\lim_{j\rightarrow\infty}\int_{\R}\int_{\R}W'(X^{k_j}(y,\xi)-X^{k_j}(z,\xi)) h(X^{k_j}(y,\tau)) d\rho^{k_j}_0(z)d\rho^{k_j}_0(y)\\
&\quad =\int_{\R}\int_{\R}W'(X(y,\xi)-X(z,\xi)) h(X(y,\tau)) d\rho_0(z)d\rho_0(y)
\end{align}
for each $\xi\in [0,\tau]$. What's more, \eqref{SimpleConvBoundWprime} implies 
\be
\left|\int_{\R}\int_{\R}W'(X^{k_j}(y,\xi)-X^{k_j}(z,\xi)) h(X^{k_j}(y,\tau)) d\rho^{k_j}_0(z)d\rho^{k_j}_0(y)\right|
\le 18(Ca e^{b \tau^2})^2\left(1+2\int_{\R}y^2d\rho^{k_j}_0(y)\right)
\ee
for $\xi\in [0,\tau]$. The limit \eqref{rhoQuadraticConv} and a standard variant of dominated convergence (Theorem 1.20 in \cite{EvaGar}) together give
\begin{align}
&\lim_{j\rightarrow\infty}\int^\tau_0\int_{\R}\int_{\R}W'(X^{k_j}(y,\xi)-X^{k_j}(z,\xi)) h(X^{k_j}(y,\tau)) d\rho^{k_j}_0(z)d\rho^{k_j}_0(y)d\xi\\
&\quad =\int^\tau_0\int_{\R}\int_{\R}W'(X(y,\xi)-X(z,\xi)) h(X(y,\tau)) d\rho_0(z)d\rho_0(y)d\xi.
\end{align}

\par As a result, 
\begin{align}\label{HlimitvzeroMinus}
&\lim_{j\rightarrow\infty}\int_{\R}\left[v_0-\displaystyle\int^\tau_0(W'*\rho^{k_j}_\xi)(X^{k_j}(\xi))d\xi\right] h(X^{k_j}(\tau))d\rho^{k_j}_0\\
&\hspace{1in} =\int_{\R}\left[v_0-\displaystyle\int^\tau_0(W'*\rho_\xi)(X(\xi))d\xi\right] h(X(\tau))d\rho_0.
\end{align}
Since
\begin{align*}
&\left|\int_{\R}\left[v_0-\displaystyle\int^\tau_0(W'*\rho^{k_j}_\xi)(X^{k_j}(\xi))d\xi\right] h(X^{k_j}(\tau))d\rho^{k_j}_0\right|\\
&\quad \le \int_{\R}|h(X^{k_j}(\tau))||v_0|d\rho^{k_j}_0+ 
\tau\cdot 18(Ca e^{b \tau^2})^2\left(1+2\int_{\R}y^2d\rho^{k_j}_0(y)\right),
\end{align*}
and the integrals $\int_{\R}|h(X^{k_j}(\tau))||v_0|d\rho^{k_j}_0,\int_{\R}y^2d\rho^{k_j}_0(y)$ are bounded and converge for each $\tau\in [s,t]$ as $j\rightarrow\infty$, we likewise have 
\begin{align*}
&\lim_{j\rightarrow\infty}\int^t_s\int_{\R}\left[v_0-\displaystyle\int^\tau_0(W'*\rho^{k_j}_\xi)(X^{k_j}(\xi))d\xi\right] h(X^{k_j}(\tau))d\rho^{k_j}_0d\tau\\
&\quad =\int^t_s\int_{\R}\left[v_0-\displaystyle\int^\tau_0(W'*\rho_\xi)(X(\tau))d\xi\right] h(X(\tau))d\rho_0d\tau.
\end{align*}
The claim \eqref{SolnGoal} then follows by sending $j\rightarrow\infty$ in 
\be\label{solnconditionXkayjay}
\int^t_s\int_{\R}\dot X^{k_j}(\tau) h(X^{k_j}(\tau))d\rho^{k_j}_0d\tau=\int^t_s\int_{\R}\left[v_0-\displaystyle\int^\tau_0(W'*\rho^{k_j}_\xi)(X^{k_j}(\xi))d\xi\right]h(X^{k_j}(\tau))d\rho^{k_j}_0d\tau.
\ee

\par  3. \underline{Properties of the solution}:  Assertions $(ii)$ and $(iii)$ of this claim follow from parts $(i)$ and $(ii)$ of Proposition \ref{StrongCompactXkayjay}, respectively.  So 
we will now focus on verifying assertion $(i)$ which states that
\be
E(\tau):=\int_{\R}\frac{1}{2}\dot X(\tau)^2d\rho_0+\int_{\R}\int_{\R}\frac{1}{2}W(X(y,\tau)-X(z,\tau))d\rho_0(y)d\rho_0(z)
\ee
is essentially nonincreasing on $[0,\infty)$.  Recall that 
\be
E^j(\tau):=\int_{\R}\frac{1}{2}\dot X^{k_j}(\tau+)^2d\rho^{k_j}_0+\int_{\R}\int_{\R}\frac{1}{2}W(X^{k_j}(y,\tau)-X^{k_j}(z,\tau))d\rho^{k_j}_0(y)d\rho^{k_j}_0(z)
\ee
is nonincreasing by part $(iii)$ of Proposition \ref{StrongCompactXkayjay}.

\par As $X^{k_j}$ solves the pressureless Euler flow equation, we can integrate by parts to find
\begin{align*}
\int^t_s\int_{\R}\dot X^{k_j}(\tau)^2d\rho^{k_j}_0d\tau&= \int^t_s\int_{\R}\dot X^{k_j}(\tau)\left[v_0-\int^\tau_0(W'*\rho^{k_j}_s)(X^{k_j}(s))ds\right]d\rho^{k_j}_0d\tau\\
& = \int_{\R}\left(\int^t_s\dot X^{k_j}(\tau)\left[v_0-\int^\tau_0(W'*\rho^{k_j}_s)(X^{k_j}(s))ds\right]d\tau \right)d\rho^{k_j}_0\\
& = \left.\int_{\R}X^{k_j}(\tau)\left[v_0-\int^\tau_0(W'*\rho^{k_j}_s)(X^{k_j}(s))ds\right] d\rho^{k_j}_0 \right|^{\tau=t}_{\tau=s}\\ 
& \hspace{1in} +\int^t_s\int_{\R} X^{k_j}(\tau)(W'*\rho^{k_j}_\tau)(X^{k_j}(\tau)) d\rho^{k_j}_0d\tau.
\end{align*}
The limit \eqref{FirstThingWeNeedforKineticEnergyLimit} with $\tau=t,s$ and $h=\id$ and the limit \eqref{HlimitvzeroMinus} with $h=\id$ give 
\begin{align}\label{KineticEnergyLimit}
\lim_{j\rightarrow\infty}\int^t_s\int_{\R}\dot X^{k_j}(\tau)^2d\rho^{k_j}_0d\tau&=\left.\int_{\R}X(\tau)\left[v_0-\int^\tau_0(W'*\rho_s)(X(s))ds\right] d\rho_0 \right|^{\tau=t}_{\tau=s}\\ 
& \hspace{1in} +\int^t_s\int_{\R} X(\tau)(W'*\rho_\tau)(X(\tau)) d\rho_0d\tau\\
&=\int^t_s\int_{\R}\dot X(\tau)^2d\rho_0d\tau
\end{align}
for each $0\le s\le t$. 

\par Next, we claim that 
\be\label{SecondPartEnergyConv}
\lim_{j\rightarrow\infty}\int_{\R}\int_{\R}W(X^{k_j}(y,\tau)-X^{k_j}(z,\tau))d\rho^{k_j}_0(y)d\rho^{k_j}_0(z)=\int_{\R}\int_{\R}W(X(y,\tau)-X(z,\tau))d\rho_0(y)d\rho_0(z)
\ee
for each $\tau\ge 0$. Note that $W$ grows at most quadratically. Indeed, 
\be
W(x)\ge W(0)+W'(0)x-\frac{c}{2}x^2
\ee
for $x\in \R$ as $W$ is semiconvex. Likewise
\be
W(0)\ge W(x)+W'(x)(0-x)-c\frac{(0-x)^2}{2},
\ee
and so 
\be
W(x)\le W(0)+xW'(x)+\frac{c}{2}x^2.
\ee
Since $W'(x)$ grows at most linearly, it must be that 
\be\label{Wquadratic}
|W(x)|\le D(1+x^2)
\ee
for some constant $D>0$. 

\par Combining \eqref{Wquadratic} and \eqref{XkayUpperBound}, it is routine to check 
\be\label{WXkayjayUpper}
|W(X^{k_j}(y,\tau)-X^{k_j}(z,\tau))|\le 12D a^2 e^{2b \tau^2}(1+y^2+z^2)
\ee
for each $j\in \N$, $y,z\in \R$, $\tau\ge 0$. Note that 
\be
\lim_{j\rightarrow\infty}W(X^{k_j}(y,\tau)-X^{k_j}(z,\tau))=W(X(y,\tau)-X(z,\tau))
\ee
locally uniformly for $y,z\in \R$ for each $\tau\ge 0$. Also observe that $\rho^{k_j}_0\times \rho^{k_j}_0\rightarrow \rho_0\times \rho_0$ narrowly in ${\cal P}(\R^2)$, and 
\be
\lim_{j\rightarrow\infty}\iint_{\R^2}(y^2+z^2)d\rho^{k_j}_0(y)d\rho^{k_j}_0(z)=\iint_{\R^2}(y^2+z^2)d\rho_0(y)d\rho_0(z),
\ee
which implies that $(y,z)\mapsto y^2+z^2$ is uniformly integrable with respect to $(\rho^{k_j}_0\times \rho^{k_j}_0)_{j\in \N}$. It follows from Lemma \ref{LemmaVariantNarrowCon} that
 \eqref{SecondPartEnergyConv} holds for each $\tau\ge 0$.

\par In view of estimate \eqref{WXkayjayUpper}, we can also apply a standard variant of dominated convergence to find 
\begin{align}\label{SecondPartEnergyConvPart2}
&\lim_{j\rightarrow\infty}\int^t_s\int_{\R}\int_{\R}W(X^{k_j}(y,\tau)-X^{k_j}(z,\tau))d\rho^{k_j}_0(y)d\rho^{k_j}_0(z)d\tau\\
&\hspace{1in} =\int^t_s\int_{\R}\int_{\R}W(X(y,\tau)-X(z,\tau))d\rho_0(y)d\rho_0(z)d\tau
\end{align}
for each $0\le s\le t$. Combining with \eqref{KineticEnergyLimit} gives 
\be\label{EjayLimit}
\lim_{j\rightarrow\infty}\int^t_sE^j(\tau)d\tau=\int^t_sE(\tau)d\tau
\ee
for each $0\le s\le t$. As $E^j:[0,\infty)\rightarrow [0,\infty)$ is uniformly bounded and nonincreasing, $E^j(t)$ also converges for each $t\ge 0$ by Helly's selection theorem (Lemma 3.3.3 in \cite{AGS}). By \eqref{EjayLimit}, $\lim_{j\rightarrow\infty}E^j(t)=E(t)$ for almost every $t\in [0,\infty)$. We then conclude that for almost every $t,s\in[0,\infty)$ with $t\ge s$
$$
E(t)=\lim_{j\rightarrow\infty}E^j(t)\le \lim_{j\rightarrow\infty}E^j(s)=E(s).
$$
\end{proof}

\subsection{Solution of the pressureless Euler equations}
We are now in position to establish the existence of a weak solution of the pressureless Euler equations \eqref{PEE} which satisfy the 
initial conditions \eqref{Init}. These types of solutions are defined as follows. 

\begin{defn}\label{WeakSolnDefn}
A narrowly continuous 
$\rho: [0,\infty)\rightarrow {\cal P}(\R); t\mapsto \rho_t$ and a Borel measurable
$v:\R\times[0,\infty)\rightarrow\R$ is a {\it weak solution pair of the pressureless Euler equations \eqref{PEE}}
which satisfies the initial conditions \eqref{Init} if the following hold. \\
$(i)$ For each $T>0$,
$$
\int^T_0\left\{\int_{\R}v^2d\rho_t+\iint_{\R^2}|W'(x-y)|d\rho_t(x)d\rho_t(y)\right\}dt<\infty.
$$ 
$(ii)$ For each $\phi\in C^\infty_c(\R\times[0,\infty))$,
$$
\int^\infty_0\int_{\R}(\partial_t\phi+v\partial_x\phi)d\rho_tdt+\int_{\R}\phi(\cdot,0)d\rho_0=0.
$$
$(iii)$ For each $\phi\in C^\infty_c(\R\times[0,\infty))$,
$$
\int^\infty_0\int_{\R}(v\partial_t\phi +v^2\partial_x\phi)d\rho_tdt+\int_{\R}\phi(\cdot,0)v_0d\rho_0=\int^\infty_0\int_{\R}\phi(W'*\rho_t)d\rho_tdt.
$$
\end{defn}

\begin{cor}\label{PEEexistCor}
There exists a weak solution pair $\rho$ and $v$ of the pressureless Euler equations \eqref{PEE} which satisfies the initial conditions \eqref{Init}.  Moreover, this solution pair has the following two properties. 
\begin{enumerate}[(i)]

\item For almost every $0\le s\le t$, 
\begin{align*}
&\int_{\R}\frac{1}{2}v(x,t)^2d\rho_t(x)+\iint_{\R^2}\frac{1}{2}W(x-y)d\rho_t(x)d\rho_t(y)\\
&\hspace{1in} \le \int_{\R}\frac{1}{2}v(x,s)^2d\rho_s(x)+\iint_{\R^2}\frac{1}{2}W(x-y)d\rho_s(x)d\rho_s(y).
\end{align*}

\item For almost every $t>0$ and $\rho_t$ almost every $x,y\in \R$,
\be
(v(x,t)-v(y,t))(x-y)\le \frac{\sqrt{c}}{\tanh(\sqrt{c}t)}(x-y)^2.
\ee
\end{enumerate}
\end{cor}
\begin{proof}
\underline{1. Specifying $v$}:  Let $X$ denote the solution of the pressureless Euler flow equation \eqref{FlowMapEqn} which satisfies \eqref{xInit} as described in  Theorem \ref{mainThm} and define
$$
\nu(S):=\int^\infty_0\int_{\R}\left[v_0-\displaystyle\int^t_0(W'*\rho_s)(X(s))ds\right]\chi_{S}(X(t),t) d\rho_0dt
$$
for Borel $S\subset \R\times [0,\infty)$. This clearly defines a signed Borel measure on $\R\times [0,\infty)$, and it is not hard to check that $\nu$ is sigma finite. Let us also set 
$$
\mu(S):=\int^\infty_0\int_{\R}\chi_{S} d\rho_tdt=\int^\infty_0\int_{\R}\chi_{S}(X(t),t) d\rho_0dt
$$
for Borel $S\subset \R\times [0,\infty)$. It is clear that $\mu$ it is also sigma finite and that $\nu$ is absolutely continuous with respect to $\mu$.

\par By the Radon-Nikodym theorem, there is a Borel measurable $v:\R\times[0,\infty)\rightarrow \R$
such that 
$$
\nu(S)=\int_S vd\mu.
$$
Note in particular that 
\begin{align*}
\int^\infty_0\int_{\R}\phi d\nu&=\int^\infty_0\int_{\R}\phi(X(t),t) \left[v_0-\displaystyle\int^t_0(W'*\rho_s)(X(s))ds\right] d\rho_0dt\\
&=\int^\infty_0\int_{\R}\phi(X(t),t)\dot X(t) d\rho_0dt\\
&=\int^\infty_0\int_{\R}\phi vd\mu\\
&=\int^\infty_0\int_{\R}\phi(X(t),t) v(X(t),t)d\rho_0dt\\
\end{align*}
for each continuous $\phi:\R\times [0,\infty)\rightarrow \R$ with compact support. It follows that 
\be\label{FindingV}
\dot X(t)=v(X(t),t)=\E_{\rho_0}\left[v_0-\displaystyle\int^t_0(W'*\rho_s)(X(s))ds\bigg| X(t)\right]\quad \text{a.e.}\; t>0
\ee
$\rho_0$ almost everywhere.

\par \underline{2. Integrability}: Since $X: [0,\infty)\rightarrow L^2(\rho_0)$ is locally Lipschitz, 
\be
\int^T_0\int_{\R}v^2d\rho_tdt=\int^T_0\int_{\R}\dot X(t)^2d\rho_0dt<\infty
\ee
for each $T>0$. Lipschitz continuity also implies that $\int_{\R}|X(t)|d\rho_0$ is bounded on $[0,T]$, so 
\begin{align*}
\int^T_0\iint_{\R^2}|W'(x-y)|d\rho_t(x)d\rho_t(y)dt&=\int^T_0\iint_{\R^2}|W'(X(w,t)-X(z,t))|d\rho_0(w)d\rho_0(z)dt\\
&\le C\int^T_0\iint_{\R^2}(|X(w,t)-X(z,t)|+1)d\rho_0(w)d\rho_0(z)dt\\
&\le C\int^T_0\iint_{\R^2}(|X(w,t)|+|X(z,t)|+1)d\rho_0(w)d\rho_0(z)dt\\
&\le 2C\left(\int^T_0\int_{\R}|X(t)|d\rho_0dt+1\right)<\infty.
\end{align*}
Thus, $\rho$ and $v$ satisfy part $(i)$ of Definition \ref{WeakSolnDefn}.

\par \underline{3. Weak solution property}: Suppose $\phi\in C^\infty_c(\R\times [0,\infty))$ 
\begin{align*}
&\int^\infty_0\int_\R\left(\partial_t\phi+v\partial_x\phi\right)d\rho_tdt\\
&\quad=\int^\infty_0\int_\R\left(\partial_t\phi(X(t),t)+v(X(t),t)\partial_x\phi(X(t),t)\right)d\rho_0dt\\
&\quad=\int^\infty_0\int_\R\left(\partial_t\phi(X(t),t)+\dot X(t)\partial_x\phi(X(t),t)\right)d\rho_0dt\\
&\quad=\int^\infty_0\int_\R\frac{d}{dt}\phi(X(t),t)d\rho_0dt\\
&\quad=\int_\R\int^\infty_0\frac{d}{dt}\phi(X(t),t)dt d\rho_0\\
&\quad=-\int_\R\phi(X(0),0) d\rho_0\\
&\quad=-\int_\R\phi(\cdot,0)d\rho_0.
\end{align*}
This proves part $(ii)$ of Definition \ref{WeakSolnDefn}.   As for part $(iii)$ that definition, 

\begin{align*}
&\int^\infty_0\int_\R\left(\partial_t\phi+v\partial_x\phi\right)vd\rho_tdt\\
&\quad=\int^\infty_0\int_\R\left(\partial_t\phi(X(t),t)+v(X(t),t)\partial_x\phi(X(t),t)\right)v(X(t),t)d\rho_0dt\\
&\quad=\int^\infty_0\int_\R\left(\partial_t\phi(X(t),t)+v(X(t),t)\partial_x\phi(X(t),t)\right)\dot X(t)d\rho_0dt\\
&\quad=\int^\infty_0\int_\R\frac{d}{dt}\phi(X(t),t) \dot X(t)d\rho_0dt\\
&\quad=\int^\infty_0\int_\R\frac{d}{dt}\phi(X(t),t) \left[v_0-\displaystyle\int^t_0(W'*\rho_s)(X(s))ds\right]d\rho_0dt\\
&\quad=\int_\R\int^\infty_0\frac{d}{dt}\phi(X(t),t) \left[v_0-\displaystyle\int^t_0(W'*\rho_s)(X(s))ds\right]dt d\rho_0\\
&\quad=\int^\infty_0\int_\R\phi(X(t),t)(W'*\rho_t)(X(t)) d\rho_0 dt -\int_\R\phi(X(0),0)v_0 d\rho_0\\
&\quad=\int^\infty_0\int_\R\phi(W'*\rho_t) d\rho_t dt -\int_\R\phi(\cdot,0)v_0 d\rho_0.
\end{align*}

 \underline{4. Nonincreasing energy and entropy inequality}: Since 
 \begin{align*}
{\cal E}(\tau)&:=\int_{\R}\frac{1}{2}v(x,\tau)^2d\rho_\tau(x)+\iint_{\R^2}\frac{1}{2}W(x-y)d\rho_\tau(x)d\rho_\tau(y)\\
&=\int_{\R}\frac{1}{2}v(X(\tau),\tau)^2d\rho_0+\iint_{\R^2}\frac{1}{2}W(X(y,\tau)-X(z,\tau))d\rho_0(y)d\rho_0(z)\\
&=\int_{\R}\frac{1}{2}\dot X(\tau)^2d\rho_0+\iint_{\R^2}\frac{1}{2}W(X(y,\tau)-X(z,\tau))d\rho_0(y)d\rho_0(z)
\end{align*}
for almost every $\tau\ge 0$, ${\cal E}(t)\le {\cal E}(s)$ for almost every $t,s\in[0,\infty)$ with $s\le t$.  Here, of course, we are
employing conclusion $(i)$ of Theorem \ref{mainThm}. 

\par Recall that $[0,\infty)\ni t\mapsto X(y,t)$ is absolutely continuous on any compact interval within $[0,\infty)$ for $\rho_0$ almost every $y\in \R$. Let us denote this set of $y$ as $Q\subset \R$, and we emphasize that $Q$ is $\rho_0$ measurable and $\rho_0(Q)=1$. By conclusion $(iii)$ of Theorem \ref{mainThm},
\begin{align*}
0&\ge \frac{d}{dt}\frac{(X(y,t)-X(z,t))^2}{\sinh(\sqrt{c}t)^2}\\
&= \frac{2(X(y,t)-X(z,t))(\partial_tX(y,t)-\partial_tX(z,t))}{\sinh(\sqrt{c}t)^2}-2\frac{\sqrt{c}\cosh(\sqrt{c}t)}{\sinh(\sqrt{c}t)^3}(X(y,t)-X(z,t))^2\\
&= \frac{2}{\sinh(\sqrt{c}t)^2}\left[(X(y,t)-X(z,t))(v(X(y,t),t)-v(X(z,t),t))-\frac{\sqrt{c}}{\tanh(\sqrt{c}t)}(X(y,t)-X(z,t))^2 \right]
\end{align*} 
for Lebesgue almost every $t>0$ and $y,z\in Q$. 

\par As a result, we have proved part $(ii)$ of this assertion for $x,y$ belonging to the image of $Q$ under $X(t)$
\be
X(t)(Q)=\{X(y,t)\in \R: y\in Q\}
\ee
for almost every $t>0$. Without loss of generality, we may suppose $Q$ is a countable union of closed sets (Theorem 1.19 of \cite{Folland}).  By part $(ii)$ of Theorem \ref{mainThm}, we may as well assume $X(t)$ is continuous on $\R$. It follows that $X(t)(Q)$ is Borel measurable (Proposition A.1 in \cite{MR4007615}). As  $Q\subset X(t)^{-1}[X(t)(Q)]$,
\be
\rho_t(X(t)(Q))=\rho_0(X(t)^{-1}[X(t)(Q)])\ge \rho_0(Q)=1.
\ee
We conclude that part $(ii)$ of this assertion holds for Lebesgue almost every $t>0$ and $x,y$ belonging to a Borel set of full measure for $\rho_t$.
\end{proof}

\section{Euler-Poisson equations in 1D}\label{EPsect}
For the remainder of this paper, we will assume
\be
W(x)=|x|, \quad x\in \R.
\ee
As $W$ is not continuously differentiable,
we will also consider the closely related interaction potential
\be
W_\epsilon(x)=(x^2+\epsilon^2)^{1/2}, \quad x\in \R
\ee
for $\epsilon>0$ and small. Clearly, 
\be\label{WepsUniform}
W(x)\le W_\epsilon(x)\le W(x)+\epsilon,\quad x\in \R.
\ee
Also observe that $W_\epsilon$ is convex, even, and continuously differentiable with 
\be
|W_\epsilon'(x)|\le 1.
\ee
\par Therefore, there is a locally Lipschitz continuous $X^\epsilon:[0,\infty)\rightarrow L^2(\rho_0)$ which satisfies
\be\label{XepsilonEquation}
\begin{cases}
\dot X^\epsilon(t)=\E_{\rho_0}\left[v_0-\displaystyle\int^t_0(W_\epsilon'*\rho^\epsilon_s)(X^\epsilon(s))ds\bigg| X^\epsilon(t)\right], \quad a.e.\;t\ge 0\\
X^\epsilon(0)=\id
\end{cases}
\ee
for each $\epsilon>0$.  Here 
$$
\rho^\epsilon_t:=X^\epsilon(t)_{\#}\rho_0, \quad t\ge 0.
$$
We will argue below that there is a sequence of positive numbers $\epsilon_k\rightarrow 0$ and $(X^{\epsilon_k})_{k\in\N}$ which converges in a strong sense to 
a solution $X$ of the Euler-Poisson flow equation \eqref{FlowMapEqnEP} which satisfies the initial condition \eqref{xInit}.   We will then make a final remark on the existence of weak solution pairs to the Euler-Poisson system \eqref{EP1D}.

\subsection{A strongly convergent subsequence}
Let us begin by recalling a few facts we have already established for $X^\epsilon$.
\\
\par  {\bf Lipschitz continuity in time}. In view of Theorem \ref{mainThm}, $X^\epsilon$ satisfies
\begin{align}
\int_{\R}\frac{1}{2}(\dot X^\epsilon(t))^2d\rho_0&\le \int_{\R}\frac{1}{2}(\dot X^\epsilon(t))^2d\rho_0+\iint_{\R^2}\frac{1}{2}W_\epsilon(X^\epsilon(y,t)-X^\epsilon(z,t))d\rho_0(y)d\rho_0(z)\\
&\le \int_{\R}\frac{1}{2}v_0^2d\rho_0+\iint_{\R^2}\frac{1}{2}W_\epsilon(y-z)d\rho_0(y)d\rho_0(z)\\
&\le \int_{\R}\frac{1}{2}v_0^2d\rho_0+\iint_{\R^2}\frac{1}{2}(|y-z|+\epsilon)d\rho_0(y)d\rho_0(z)\\
&= \int_{\R}\frac{1}{2}v_0^2d\rho_0+\iint_{\R^2}\frac{1}{2}|y-z|d\rho_0(y)d\rho_0(z)+\frac{1}{2}\epsilon
\end{align}
for almost every $t\ge 0$. Therefore, $X^\epsilon:[0,\infty)\rightarrow L^2(\rho_0)$ is uniformly Lipschitz continuous. 
\\
\par {\bf Uniform spatial continuity}. Theorem \ref{mainThm} also gives
\be
0\le X^\epsilon(y,t)-X^\epsilon(z,t)\le y-z +t\int^y_z|v_0'(x)|dx
\ee
for each $y,z\in\text{supp}(\rho_0)$ with $y\ge z$ and for each $t\ge 0$.  As a result, we may was well suppose $X^\epsilon(t):\R\rightarrow \R$ is uniformly continuous and satisfies 
\be
|X^\epsilon(y,t)-X^\epsilon(z,t)|\le |y-z| +t\omega(|y-z|)
\ee
for every $y,z\in \R$ and $t\ge 0$. Here $\omega$ is the modulus of continuity defined in \eqref{modulusOmega}.
\\
\par {\bf Quantitative stickiness}.  For $0<s\le t$, there is a function $f^\epsilon_{t,s}:\R\rightarrow \R$ which satisfies 
\be\label{SimpleLipCond}
|f^\epsilon_{t,s}(x)-f^\epsilon_{t,s}(y)|\le \frac{t}{s}|x-y|,\quad x,y\in \R
\ee
such that 
\be\label{SemiGroupPropXeps}
X^\epsilon(y,t)=f^\epsilon_{t,s}(X^\epsilon(y,s))
\ee
for $y\in\text{supp}(\rho_0)$. In particular, 
\be
\frac{1}{t}|X^\epsilon(y,t)-X^\epsilon(z,t)|\le \frac{1}{s}|X^\epsilon(y,s)-X^\epsilon(z,s)|
\ee
for $y,z\in\text{supp}(\rho_0)$ and $0<s\le t$.  
\\
\par We can use this Lipschitz continuity in time, uniform spatial continuity and quantitative stickiness of $(X^\epsilon)_{\epsilon>0}$, along with the arguments we used to prove Proposition \ref{StrongCompactXkayjay} and Corollary \ref{dotXfunctionofX}, to establish the following assertion. 
 
\begin{prop}\label{XepsKayCompactness}
There is a sequence of positive numbers $\epsilon_k\rightarrow 0$ and a Lipschitz $X: [0,\infty)\rightarrow L^2(\rho_0)$ such that for each $t\ge 0$,
\be
X^{\epsilon_k}(t)\rightarrow X(t)
\ee
locally uniformly on $\R$ and
\be
\lim_{j\rightarrow\infty}\int_{\R}h(\textup{id}_\R,X^{\epsilon_k}(t))d\rho_0=\int_{\R}h(\textup{id}_\R,X(t))d\rho_0
\ee
for continuous $h:\R^2\rightarrow \R$ with
\be
\sup_{(x,y)\in \R^2}\frac{|h(x,y)|}{1+x^2+y^2}<\infty.
\ee
In addition, $X$ has the following properties.  

\begin{enumerate}[(i)]

\item For $t\ge 0$ and $y,z\in\textup{supp}(\rho_0)$ with $y\le z$, 
$$
0\le X(z,t)-X(y,t)\le z-y+t\int^z_y|v_0'(x)|dx
$$

\item For each $0<s\le t$ and $y,z\in\textup{supp}(\rho_0)$
\be
\frac{1}{t}|X(y,t)-X(z,t)|\le \frac{1}{s}|X(y,s)-X(z,s)|
\ee

\item For each $0<s\le t$, $(f^{\epsilon_k}_{t,s})_{k\in \N}$ has a subsequence which converges locally uniformly to a function $f_{t,s}:\R\rightarrow \R$ which satisfies the Lipschitz condition \eqref{SimpleLipCond} 
and
$$
X(y,t)=f_{t,s}(X(y,s))
$$
for $y\in\textup{supp}(\rho_0)$.

\item For almost every  $t>0$, there is a Borel function $u:\R\rightarrow \R$ for which
$$
\dot X(t)=u(X(t))
$$
for $\rho_0$ almost everywhere. 
\end{enumerate}
\end{prop}
We also have the following immediate corollary which can be proved the same way that we justified \eqref{WeakConvDotXkayjay}. 
\begin{cor}\label{WeakconvDerXepsKay}
Suppose $g:\R\rightarrow \R$ is continuous and
\be
\sup_{x\in \R}\frac{|g(x)|}{1+|x|}<\infty.
\ee
Then
\be
\lim_{k\rightarrow\infty}\int^t_s\int_{\R}\dot X^{\epsilon_k}(\tau) g(X^{\epsilon_k}(\tau))d\rho_0d\tau=\int^t_s\int_{\R}\dot X(\tau) g(X(\tau))d\rho_0d\tau
\ee
for $0\le s\le t$. 
\end{cor}
This is as far as we can go with the convergence arguments we used to establish Theorem \ref{mainThm}.  We will need to identify another mechanism
which will allow us to pass to the limit in the term with $W_\epsilon'$ in equation \eqref{XepsilonEquation}. This will be the topic of the following subsection.

\subsection{A convergence lemma}
Let us recall the definition of $\sgn$
\be
\sgn(x)=
\begin{cases}
1,\quad &x>0\\
0,\quad &x=0\\
-1,\quad &x<0.
\end{cases}
\ee
We will also fix a sequence $(\mu^k)_{k\in \N}\subset{\cal P}(\R)$ which converges narrowly to $\mu\in {\cal P}(\R)$ and additionally satisfies
\be\label{myookayConvStrong}
\lim_{k\rightarrow\infty}\int_{\R}x^2d\mu^k(x)=\int_{\R}x^2d\mu(x).
\ee
The central assertion of this subsection is as follows.
\begin{lem}\label{SubdifferentialConProp}
Suppose $g:\R\rightarrow \R$ is continuous and
\be\label{geeGrowAtMostLinearly}
\sup_{x\in \R}\frac{|g(x)|}{1+|x|}<\infty.
\ee
Then
\be
\lim_{k\rightarrow\infty}\iint_{\R^2}W'_{\epsilon_k}(x-y)g(x)d\mu^k(x)d\mu^k(y)
=\iint_{\R^2}\textup{\sgn}(x-y)g(x)d\mu(x)d\mu(y).
\ee
\end{lem}

\par We will first verify an elementary observation, which is ultimately due to the convexity of the absolute value function. In particular, we will employ
\be\label{sgnSubdiff}
|y|\ge |x|+\sgn(x)(y-x)
\ee
for each $x,y\in \R$.

\begin{lem}\label{SubdiffLemma}
The following are equivalent for $\xi\in L^2(\mu)$. \\
(i) For $\mu$ almost every $x\in \R$,
$$
\xi(x)=\textup{\sgn}*\mu(x).
$$
(ii) For each continuous $g:\R\rightarrow \R$ which satisfies \eqref{geeGrowAtMostLinearly},
$$
\iint_{\R^2}\frac{1}{2}|x-y+(g(x)-g(y))|d\mu(x)d\mu(y)\ge \iint_{\R^2}\frac{1}{2}|x-y|d\mu(x)d\mu(y)+\int_{\R}\xi gd\mu.
$$
\end{lem}

\begin{proof}
Suppose $g:\R\rightarrow \R$ is a continuous function which grows at most linearly as $|x|\rightarrow\infty$. 

\par $(i)\Longrightarrow (ii)$ Employing \eqref{sgnSubdiff} and noting that $\sgn$ is odd gives
\begin{align*}
&\iint_{\R^2}\frac{1}{2}|x-y+(g(x)-g(y))|d\mu(x)d\mu(y)\\
&\quad\quad \ge \iint_{\R^2}\frac{1}{2}|x-y|d\mu(x)d\mu(y)+\frac{1}{2}\iint_{\R^2}\sgn(x-y)(g(x)-g(y))d\mu(x)d\mu(y)\\
&\quad\quad=\iint_{\R^2}\frac{1}{2}|x-y|d\mu(x)d\mu(y)+\frac{1}{2}\iint_{\R^2}\sgn(x-y)g(x)d\mu(x)d\mu(y)\\
&\hspace{1in} -\frac{1}{2}\iint_{\R^2}\sgn(x-y)g(y)d\mu(x)d\mu(y)\\
&\quad\quad= \iint_{\R^2}\frac{1}{2}|x-y|d\mu(x)d\mu(y)+\frac{1}{2}\iint_{\R^2}\sgn(x-y)g(x)d\mu(x)d\mu(y)\\
&\hspace{1in} +\frac{1}{2}\iint_{\R^2}\sgn(y-x)g(y)d\mu(x)d\mu(y)\\
&\quad\quad= \iint_{\R^2}\frac{1}{2}|x-y|d\mu(x)d\mu(y)+\int_{\R}\left(\int_{\R}\sgn(x-y)d\mu(y)\right)g(y)d\mu(x)\\
&\quad\quad= \iint_{\R^2}\frac{1}{2}|x-y|d\mu(x)d\mu(y)+\int_{\R}(\sgn*\mu) gd\mu.
\end{align*}

\par $(ii)\Longrightarrow (i)$ By assumption,
\be\label{PhiConvexIneq}
\frac{1}{2}\int_{\R}\frac{|x-y+t(g(x)-g(y))|-|x-y|}{t}d\mu(x)d\mu(y)\ge \int_{\R}g \xi d\mu.
\ee
Also notice 
$$
\left|\frac{|x-y+t(g(x)-g(y))|-|x-y|}{t}\right|\le |g(x)-g(y)|
$$
and 
$$
\lim_{t\rightarrow 0^+}\frac{|x-y+t(g(x)-g(y))|-|x-y|}{t}=\sgn(x-y)(g(x)-g(y)).
$$
for $x,y\in \R$. By dominated convergence, we can send $t\rightarrow 0^+$ in \eqref{PhiConvexIneq} to find 
$$
\int_{\R}g \xi d\mu\le \frac{1}{2}\int_{\R}\int_{\R}\sgn(x-y)(g(x)-g(y))d\mu(x)d\mu(y)=\int_{\R}(\sgn*\mu) gd\mu.
$$
Replacing $g$ with $-g$ gives
$$
\int_{\R}g (\xi -\sgn*\mu)d\mu=0.
$$
As $g$ is arbitrary, $\xi - \sgn*\mu$ vanishes $\mu$ almost everywhere. 
\end{proof}

\begin{proof}[Proof of Lemma \ref{SubdifferentialConProp}]
Using the same method to prove $(i)\Longrightarrow (ii)$ in Lemma \ref{SubdiffLemma}, we find 
\begin{align*}
&\iint_{\R^2}\frac{1}{2}W_{\epsilon_k}(x-y+(g(x)-g(y)))d\mu^k(x)d\mu^k(y)\\
&\hspace{1in}\ge \iint_{\R^2}\frac{1}{2}W_{\epsilon_k}(x-y)d\mu^k(x)d\mu^k(y)+\int_{\R}(W_{\epsilon_k}'*\mu^k) gd\mu^k
\end{align*}
for each continuous and at most linearly growing $g:\R\rightarrow \R$ . And by \eqref{WepsUniform},  
\begin{align}\label{approximateSubdiffLim}
&\frac{1}{2}\epsilon_k+\iint_{\R^2}\frac{1}{2}|x-y+(g(x)-g(y))|d\mu^k(x)d\mu^k(y)\\
&\hspace{1in}\ge \iint_{\R^2}\frac{1}{2}|x-y|d\mu^k(x)d\mu^k(y)+\int_{\R}(W_{\epsilon_k}'*\mu^k) gd\mu^k.
\end{align}

\par Since $|W'_{\epsilon_k}|\le 1$,
$$
|(W_{\epsilon_k}'*\mu^k)(z)|\le 1, \quad z\in \R. 
$$ 
Combining this fact with \eqref{myookayConvStrong} provides a subsequence $(W_{\epsilon_{k_j}}'*\mu^{k_j})_{j\in \N}$  and $\xi\in L^2(\mu)$ such that 
$$
\lim_{j\rightarrow\infty}\int_{\R}(W_{\epsilon_{k_j}}'*\mu^{k_j}) gd\mu^{k_j}=\int_{\R}\xi gd\mu
$$
for each continuous and at most linearly growing $g:\R\rightarrow \R$ (Theorem 5.4.4 of \cite{AGS}).  Sending $k=k_j\rightarrow\infty$ in \eqref{approximateSubdiffLim} gives
$$
\iint_{\R^2}\frac{1}{2}|x-y+(g(x)-g(y))|d\mu(x)d\mu(y)\ge \iint_{\R^2}\frac{1}{2}|x-y|d\mu(x)d\mu(y)+\int_{\R}\xi gd\mu.
$$
Lemma \ref{SubdiffLemma} implies 
$$
\xi=\sgn*\mu. 
$$
Since this limit is independent of the subsequence, 
\begin{align*}
\lim_{k\rightarrow\infty}\iint_{\R^2}W'_{\epsilon_k}(x-y)g(x)d\mu^k(x)d\mu^k(y)&=\lim_{k\rightarrow\infty}\int_{\R}(W_{\epsilon_k}'*\mu^k) gd\mu^{k}\\
&=\int_{\R}(\sgn*\mu) gd\mu\\
&=\iint_{\R^2}\sgn(x-y)g(x)d\mu(x)d\mu(y)
\end{align*}
for each continuous $g:\R\rightarrow \R$ which satisfies \eqref{SubdifferentialConProp}. 
\end{proof}

We will actually need a minor refinement of Lemma \ref{SubdiffLemma} in our proof of Theorem \ref{secondThm}. 
\begin{cor}\label{SubdifferentialConProp}
Suppose $(g^k)_{k\in \N}$ is a sequence of continuous functions on $\R$ which satisfies 
\be\label{uniformLinearGrowthGeeKay}
\sup_{x\in \R}\frac{|g^k(x)|}{1+|x|}\le C
\ee
for some $C$ and which converges locally uniformly to $g: \R\rightarrow \R$. Then 
\be
\lim_{k\rightarrow\infty}\iint_{\R^2}W'_{\epsilon_k}(x-y)g^k(x)d\mu^k(x)d\mu^k(y)=\iint_{\R^2}\textup{\sgn}(x-y)g(x)d\mu(x)d\mu(y).
\ee
\end{cor}
\begin{proof}
As $g^k(x)\rightarrow g(x)$ for each $x\in \R$, we also have 
$$
\sup_{x\in \R}\frac{|g(x)|}{1+|x|}\le C.
$$
 Fix $\delta>0$ and choose a compact interval $K_\delta\subset \R$ such that 
\be\label{ChoiceofConstantMyooKay}
\int_{\R\setminus K_\delta}(1+|x|)d\mu^k(x)\le \frac{\delta}{2C} 
\ee
for $k\in \N$; such an interval exists as $1+|x|$ is uniformly integrable with respect to $(\mu^k)_{k\in \N}$ by assumption \eqref{myookayConvStrong}.   In view of Lemma \ref{SubdiffLemma}, 
\begin{align*}
\iint_{\R^2}W'_{\epsilon_k}(x-y)g^k(x)d\mu^k(x)d\mu^k(y)&=\int_{\R}(W'_{\epsilon_k}*\mu^k)g^kd\mu^k \\
&=\int_{\R}(W'_{\epsilon_k}*\mu^k)gd\mu^k +\int_{\R}(W'_{\epsilon_k}*\mu^k)(g^k-g)d\mu^k\\
&=\int_{\R}(\textup{\sgn}*\mu) gd\mu + o(1)+\int_{\R}(W'_{\epsilon_k}*\mu^k)(g^k-g)d\mu^k\\
&=\int_{\R}(\textup{\sgn}*\mu) gd\mu + o(1)+\int_{K_\delta}(W'_{\epsilon_k}*\mu^k)(g^k-g)d\mu^k\\
&\hspace{1in}+\int_{\R\setminus K_\delta}(W'_{\epsilon_k}*\mu^k)(g^k-g)d\mu^k\\
&=\iint_{\R^2}\textup{\sgn}(x-y)g(x)d\mu(x)d\mu(y) + o(1)\\
&\quad +\int_{K_\delta}(W'_{\epsilon_k}*\mu^k)(g^k-g)d\mu^k+\int_{\R\setminus K_\delta}(W'_{\epsilon_k}*\mu^k)(g^k-g)d\mu^k
\end{align*}
as $k\rightarrow\infty$.

\par Observe  
$$
\left|\int_{K_\delta}(W'_{\epsilon_k}*\mu^k)(g^k-g)d\mu^k\right|\le \max_{K_\delta}|g^k-g|.
$$
And by \eqref{ChoiceofConstantMyooKay},
$$
\left|\int_{\R\setminus K_\delta}(W'_{\epsilon_k}*\mu^k)(g^k-g)d\mu^k\right|\le 2C\int_{\R\setminus K_\delta}(1+|x|)d\mu^k(x)\le \delta.
$$
As a result, 
$$
\limsup_{k\rightarrow\infty}\left|\iint_{\R^2}W'_{\epsilon_k}(x-y)g^k(x)d\mu^k(x)d\mu^k(y)-\iint_{\R^2}\textup{\sgn}(x-y)g(x)d\mu(x)d\mu(y)\right|\le \delta. 
$$
The claim follows as $\delta>0$ was arbitrarily chosen. 
\end{proof}

\subsection{Solution of the flow equation}
This subsection is dedicated to the proof of Theorem \ref{secondThm}. Here, we will show that the mapping $X$ obtained in Proposition \ref{XepsKayCompactness} is 
a solution flow equation \eqref{FlowMapEqnEP} which has all of the required properties. First note that since $X^{\epsilon_k}(0)=\id$ and $X^{\epsilon_k}(0)\rightarrow X(0)$
in $L^2(\rho_0)$ as $k\rightarrow\infty$, then $X(0)=\id$.  Next we claim that $X$ satisfies flow equation \eqref{FlowMapEqnEP}. It suffices to let $0\le s\le t$, fix a continuous $h:\R\rightarrow \R$ which satisfies 
$$
\sup_{x\in \R}\frac{|h(x)|}{1+|x|}\le C,
$$
and show
\be\label{SolnGoalEP}
\int^t_s\int_{\R}\dot X(\tau) h(X(\tau))d\rho_0d\tau=\int^t_s\int_{\R}\left[v_0-\displaystyle\int^\tau_0(\sgn*\rho_\xi)(X(\xi))d\xi\right]h(X(t))d\rho_0d\tau.
\ee
Once we establish this identity, parts $(i), (ii)$, and $(iii)$ of  Theorem \ref{secondThm} would follow by minor variations of the arguments we gave in our proof of 
Theorem \ref{mainThm}.

\par To this end, we recall that for each $k\in N$, 
\be\label{solnconditionXepskay}
\int^t_s\int_{\R}\dot X^{\epsilon_k}(\tau) h(X^{\epsilon_k}(\tau))d\rho_0d\tau=\int^t_s\int_{\R}\left[v_0-\displaystyle\int^\tau_0(W_{\epsilon_k}'*\rho^{\epsilon_k}_\xi)(X^{\epsilon_k}(\xi))d\xi\right]h(X^{\epsilon_k}(\tau))d\rho_0d\tau.
\ee
Moreover, Proposition \ref{XepsKayCompactness} implies 
\be\label{FinalRhoNarrowConv}
\rho_t=X(t)_{\#}\rho_0=\lim_{k\rightarrow\infty}X^{\epsilon_k}(t)_{\#}\rho_0=\lim_{k\rightarrow\infty}\rho^{\epsilon_k}_t
\ee
narrowly and 
\be\label{FinalRhoStrongConv}
\lim_{k\rightarrow\infty}\int_{\R}x^2d\rho^{\epsilon_k}_t(x)=\lim_{k\rightarrow\infty}\int_{\R}(X^{\epsilon_k}(t))^2d\rho_0=\int_{\R}(X(t))^2d\rho_0=\int_{\R}x^2d\rho_t(x)
\ee
for each $t\ge 0$.

\par Proposition \ref{XepsKayCompactness} also can be used to show 
\be
\lim_{k\rightarrow\infty}\int^t_s\int_{\R} v_0 h(X^{\epsilon_k}(\tau))d\rho_0d\tau=\int^t_s\int_{\R}v_0 h(X(t))d\rho_0d\tau.
\ee
Furthermore, 
\be
\lim_{k\rightarrow\infty}\int^t_s\int_{\R}\dot X^{\epsilon_k}(\tau) h(X^{\epsilon_k}(\tau))d\rho_0d\tau=\int^t_s\int_{\R}\dot X(\tau) h(X(\tau))d\rho_0d\tau
\ee
as noted in Corollary \ref{WeakconvDerXepsKay}. As a result, we are left to justify the limit 
\begin{align}\label{FinalSolutionLimitXepsKay}
&\lim_{k\rightarrow\infty}\int^t_s\int_{\R}\left(\int^\tau_0(W_{\epsilon_k}'*\rho^{\epsilon_k}_\xi)(X^{\epsilon_k}(\xi))d\xi \right)h(X^{\epsilon_k}(\tau))d\rho_0d\tau\\
&\hspace{1in}=
\int^t_s\int_{\R}\left(\int^\tau_0(\sgn*\rho_\xi)(X(\xi))d\xi \right)h(X(\tau))d\rho_0d\tau.
\end{align}
Then we would be able to send $k\rightarrow \infty$ in \eqref{solnconditionXepskay} to conclude \eqref{SolnGoalEP}.

\par So we will now focus on establishing \eqref{FinalSolutionLimitXepsKay}. Observe 
\begin{align}
&\int^t_s\int_{\R}\left(\int^\tau_0(W_{\epsilon_k}'*\rho^{\epsilon_k}_\xi)(X^{\epsilon_k}(\xi))d\xi \right)h(X^{\epsilon_k}(\tau))d\rho_0d\tau\\
&\hspace{1in}=\int^t_s\int^\tau_0\left[\int_{\R}(W_{\epsilon_k}'*\rho^{\epsilon_k}_\xi)(X^{\epsilon_k}(\xi)) h(X^{\epsilon_k}(\tau))d\rho_0\right]d\xi d\tau.
\end{align}
Since $W'_{\epsilon_k}$ is uniformly bounded and $h$ grows at most linearly, we just need to show 
\be\label{FinalSolutionLimitXepsKayToo}
\lim_{k\rightarrow\infty}\int_{\R}(W_{\epsilon_k}'*\rho^{\epsilon_k}_\xi)(X^{\epsilon_k}(\xi)) h(X^{\epsilon_k}(\tau))d\rho_0=
\int_{\R}(\sgn*\rho_\xi)(X(\xi)) h(X(\tau))d\rho_0
\ee
for each $\xi,\tau> 0$ with $\xi\le \tau$. For if \eqref{FinalSolutionLimitXepsKayToo} holds, \eqref{FinalSolutionLimitXepsKay} would follow from a simple application of the dominated convergence theorem.

\par In view of \eqref{SemiGroupPropXeps},  
\begin{align}
\int_{\R}(W_{\epsilon_k}'*\rho^{\epsilon_k}_\xi)(X^{\epsilon_k}(\xi)) h(X^{\epsilon_k}(\tau))d\rho_0&=\int_{\R}(W_{\epsilon_k}'*\rho^{\epsilon_k}_\xi)(X^{\epsilon_k}(\xi))\; h\circ f_{\tau,\xi}^{\epsilon_k}(X^{\epsilon_k}(\xi))d\rho_0\\
&=\int_{\R}(W_{\epsilon_k}'*\rho^{\epsilon_k}_\xi)\; h\circ f_{\tau,\xi}^{\epsilon_k} d\rho^{\epsilon_k}_\xi\\
&=\iint_{\R^2}W_{\epsilon_k}'(x-y)\; h\circ f_{\tau,\xi}^{\epsilon_k}(y) d\rho^{\epsilon_k}_\xi(x)d\rho^{\epsilon_k}_\xi(y).
\end{align}
By part $(iii)$ of Proposition \ref{XepsKayCompactness}, $f_{\tau,\xi}^{\epsilon_k}\rightarrow f_{\tau,\xi}$ locally uniformly on $\R$  (up to a subsequence that we will not relabel) and
$$
X(\tau)=f_{\tau,\xi}(X(\xi))
$$
$\rho_0$ almost everywhere. It follows that $h\circ f_{\tau,\xi}^{\epsilon_k}$ converges locally uniformly to $h\circ f_{\tau,\xi}$.  We also have the limits
\eqref{FinalRhoNarrowConv} and \eqref{FinalRhoStrongConv} for each $t=\xi$. We can then apply Corollary \ref{SubdifferentialConProp} once 
we know $h\circ f_{\tau,\xi}^{\epsilon_k}(y)$ grows at most linearly in $|y|$ in a uniform way.

\par Fix $z_0\in\text{supp}(\rho_0)$ and observe 
\begin{align}
|h\circ f_{\tau,\xi}^{\epsilon_k}(y)|&\le C(1+|f_{\tau,\xi}^{\epsilon_k}(y)|)\\
&\le C\left(1+|f_{\tau,\xi}^{\epsilon_k}(y)-f_{\tau,\xi}^{\epsilon_k}(X^{\epsilon_k}(z_0,\xi))|+|f_{\tau,\xi}^{\epsilon_k}(X^{\epsilon_k}(z_0,\xi))|\right)\\
&\le C\left(1+\frac{\tau}{\xi}|y-X^{\epsilon_k}(z_0,\xi)|+|X^{\epsilon_k}(z_0,\tau)|\right)
\end{align}
for all $y\in\R$. Since $(X^{\epsilon_k}(z_0,\xi), X^{\epsilon_k}(z_0,\tau))\rightarrow (X(z_0,\xi), X(z_0,\tau))$ as $k\rightarrow\infty$, it must be that
\be
\sup_{k\in \N}\left\{\sup_{y\in \R}\frac{|h\circ f_{\tau,\xi}^{\epsilon_k}(y)|}{1+|y|}\right\}<\infty.
\ee
Corollary \ref{SubdifferentialConProp} then gives
\begin{align}
\lim_{k\rightarrow\infty}\int_{\R}(W_{\epsilon_k}'*\rho^{\epsilon_k}_\xi)(X^{\epsilon_k}(\xi)) h(X^{\epsilon_k}(\tau))d\rho_0&=\lim_{k\rightarrow\infty}\iint_{\R^2}W_{\epsilon_k}'(x-y)\; h\circ f_{\tau,\xi}^{\epsilon_k}(y) d\rho^{\epsilon_k}_\xi(x)d\rho^{\epsilon_k}_\xi(y)\\
&=\iint_{\R^2}\sgn(x-y)\; h\circ f_{\tau,\xi}(y) d\rho_\xi(x)d\rho_\xi(y)\\
&=\int_{\R}(\sgn*\rho_\xi)(X(\xi))\; h\circ f_{\tau,\xi}(X(\xi))d\rho_0\\
&=\int_{\R}(\sgn*\rho_\xi)(X(\xi))\; h(X(\tau))d\rho_0.
\end{align}
We conclude \eqref{FinalSolutionLimitXepsKayToo} and in turn that $X$ is a solution of the flow equation \eqref{FlowMapEqnEP}.

\subsection{Solving the Euler-Poisson equations }
Weak solution pairs of the Euler-Poisson system \eqref{EP1D} which satisfy given initial conditions \eqref{Init} are defined as follows. 
\begin{defn}\label{EPWeakSolnDefn}
A narrowly continuous $\rho: [0,\infty)\rightarrow {\cal P}(\R); t\mapsto \rho_t$ and a Borel measurable
$v:\R\times[0,\infty)\rightarrow\R$ is a {\it weak solution pair of the Euler-Poisson equations \eqref{EP1D}}
which satisfies the initial conditions \eqref{Init} if the following hold. \\
$(i)$ For each $T>0$,
$$
\int^T_0\int_{\R}v^2d\rho_t dt<\infty.
$$ 
$(ii)$ For each $\phi\in C^\infty_c(\R\times[0,\infty))$,
$$
\int^\infty_0\int_{\R}(\partial_t\phi+v\partial_x\phi)d\rho_tdt+\int_{\R}\phi(\cdot,0)d\rho_0=0.
$$
$(iii)$ For each $\phi\in C^\infty_c(\R\times[0,\infty))$,
$$
\int^\infty_0\int_{\R}(v\partial_t\phi +v^2\partial_x\phi)d\rho_tdt+\int_{\R}\phi(\cdot,0)v_0d\rho_0=\int^\infty_0\int_{\R}\phi(\sgn*\rho_t)d\rho_tdt.
$$
\end{defn}

Employing the same method used to prove Corollary \ref{PEEexistCor} from Theorem \ref{mainThm}, we have the subsequent corollary to Theorem \ref{secondThm}.

\begin{cor}
There exists a weak solution pair $\rho$ and $v$ of the Euler-Poisson equations \eqref{EP1D} which satisfies the initial conditions \eqref{Init}.  Moreover, this solution pair additionally has the following features.  
\begin{enumerate}[(i)]

\item For almost every $t,s\ge 0$ with $0\le s\le t$, 
\begin{align*}
&\int_{\R}\frac{1}{2}v(x,t)^2d\rho_t(x)+\iint_{\R^2}\frac{1}{2}|x-y|d\rho_t(x)d\rho_t(y)\\
&\hspace{1in} \le \int_{\R}\frac{1}{2}v(x,s)^2d\rho_s(x)+\iint_{\R^2}\frac{1}{2}|x-y|d\rho_s(x)d\rho_s(y).
\end{align*}

\item For almost every $t>0$ and $\rho_t$ almost every $x,y\in \R$,
\be
(v(x,t)-v(y,t))(x-y)\le \frac{1}{t}(x-y)^2.
\ee
\end{enumerate}
\end{cor}

\bibliography{PEbib}{}

\begin{thebibliography}{10}

\bibitem{AGS}
L.~Ambrosio, N.~Gigli, and G.~Savar\'e.
\newblock {\em Gradient flows in metric spaces and in the space of probability
  measures}.
\newblock Lectures in Mathematics ETH Z\"urich. Birkh\"auser Verlag, Basel,
  second edition, 2008.

\bibitem{Bezard}
M.~B\'ezard.
\newblock Existence locale de solutions pour les \'equations
  d'{E}uler-{P}oisson.
\newblock {\em Japan J. Indust. Appl. Math.}, 10(3):431--450, 1993.

\bibitem{Bolley}
F.~Bolley.
\newblock Separability and completeness for the {W}asserstein distance.
\newblock In {\em S\'eminaire de probabilit\'es {XLI}}, volume 1934 of {\em
  Lecture Notes in Math.}, pages 371--377. Springer, Berlin, 2008.

\bibitem{BreGan}
Y.~Brenier, W.~Gangbo, G.~Savar\'e, and M.~Westdickenberg.
\newblock Sticky particle dynamics with interactions.
\newblock {\em J. Math. Pures Appl. (9)}, 99(5):577--617, 2013.

\bibitem{BreGre}
Y.~Brenier and E.~Grenier.
\newblock Sticky particles and scalar conservation laws.
\newblock {\em SIAM J. Numer. Anal.}, 35(6):2317--2328, 1998.

\bibitem{MR3296602}
F.~Cavalletti, M.~Sedjro, and M.~Westdickenberg.
\newblock A simple proof of global existence for the 1{D} pressureless gas
  dynamics equations.
\newblock {\em SIAM J. Math. Anal.}, 47(1):66--79, 2015.

\bibitem{Liu}
Y.~Deng, T.-P. Liu, T~Yang, and Z.~Yao.
\newblock Solutions of {E}uler-{P}oisson equations for gaseous stars.
\newblock {\em Arch. Ration. Mech. Anal.}, 164(3):261--285, 2002.

\bibitem{Dermoune}
A.~Dermoune.
\newblock Probabilistic interpretation of sticky particle model.
\newblock {\em Ann. Probab.}, 27(3):1357--1367, 1999.

\bibitem{ERykovSinai}
W.~E, Y.~Rykov, and Y.~Sinai.
\newblock Generalized variational principles, global weak solutions and
  behavior with random initial data for systems of conservation laws arising in
  adhesion particle dynamics.
\newblock {\em Comm. Math. Phys.}, 177(2):349--380, 1996.

\bibitem{EvaGar}
L.~C. Evans and R.~F. Gariepy.
\newblock {\em Measure theory and fine properties of functions}.
\newblock Textbooks in Mathematics. CRC Press, Boca Raton, FL, revised edition,
  2015.

\bibitem{Folland}
G.~Folland.
\newblock {\em Real analysis}.
\newblock Pure and Applied Mathematics (New York). John Wiley \& Sons, Inc.,
  New York, second edition, 1999.
\newblock Modern techniques and their applications, A Wiley-Interscience
  Publication.

\bibitem{GNT}
W.~Gangbo, T.~Nguyen, and A.~Tudorascu.
\newblock Euler-{P}oisson systems as action-minimizing paths in the
  {W}asserstein space.
\newblock {\em Arch. Ration. Mech. Anal.}, 192(3):419--452, 2009.

\bibitem{Guo}
Y.~Guo, L.~Han, and J.~Zhang.
\newblock Absence of shocks for one dimensional {E}uler-{P}oisson system.
\newblock {\em Arch. Ration. Mech. Anal.}, 223(3):1057--1121, 2017.

\bibitem{Gurbatov}
S.~N Gurbatov, A.~Saichev, and S.~F Shandarin.
\newblock Large-scale structure of the universe. the zeldovich approximation
  and the adhesion model.
\newblock {\em Physics-Uspekhi}, 55(3):223, 2012.

\bibitem{Hynd2}
R.~Hynd.
\newblock Sticky particles and the pressureless euler equations in one spatial
  dimension.
\newblock {\em Preprint}, 2018.

\bibitem{MR4007615}
R.~Hynd.
\newblock Lagrangian coordinates for the sticky particle system.
\newblock {\em SIAM J. Math. Anal.}, 51(5):3769--3795, 2019.

\bibitem{Jabin}
P.-E. Jabin and T.~Rey.
\newblock Hydrodynamic limit of granular gases to pressureless {E}uler in
  dimension 1.
\newblock {\em Quart. Appl. Math.}, 75(1):155--179, 2017.

\bibitem{Jin}
C.~Jin.
\newblock Well posedness for pressureless {E}uler system with a flocking
  dissipation in {W}asserstein space.
\newblock {\em Nonlinear Anal.}, 128:412--422, 2015.

\bibitem{LeFloch}
P.~LeFloch and S.~Xiang.
\newblock Existence and uniqueness results for the pressureless
  {E}uler-{P}oisson system in one spatial variable.
\newblock {\em Port. Math.}, 72(2-3):229--246, 2015.

\bibitem{Makino}
T.~Makino.
\newblock On a local existence theorem for the evolution equation of gaseous
  stars.
\newblock In {\em Patterns and waves}, volume~18 of {\em Stud. Math. Appl.},
  pages 459--479. North-Holland, Amsterdam, 1986.

\bibitem{NatSav}
L.~Natile and G.~Savar\'e.
\newblock A {W}asserstein approach to the one-dimensional sticky particle
  system.
\newblock {\em SIAM J. Math. Anal.}, 41(4):1340--1365, 2009.

\bibitem{NguTud}
T.~Nguyen and A.~Tudorascu.
\newblock Pressureless {E}uler/{E}uler-{P}oisson systems via adhesion dynamics
  and scalar conservation laws.
\newblock {\em SIAM J. Math. Anal.}, 40(2):754--775, 2008.

\bibitem{Shen}
C.~Shen.
\newblock The {R}iemann problem for the pressureless {E}uler system with the
  {C}oulomb-like friction term.
\newblock {\em IMA J. Appl. Math.}, 81(1):76--99, 2016.

\bibitem{Zeldovich}
Ya.~B. Zel'dovich.
\newblock {Gravitational instability: An Approximate theory for large density
  perturbations}.
\newblock {\em Astron. Astrophys.}, 5:84--89, 1970.

\end{thebibliography}

\bibliographystyle{plain}

\typeout{get arXiv to do 4 passes: Label(s) may have changed. Rerun}

\end{document}